\pdfoutput=1
\documentclass[a4paper,11pt]{amsart}

\usepackage[margin=1in,includehead,includefoot]{geometry}
\usepackage[utf8]{inputenc}
\usepackage[T1]{fontenc}
\usepackage{lmodern,microtype,mathtools,amssymb,amsthm,enumitem,tikz}
\usetikzlibrary{hobby,calc,intersections}
\usepackage[unicode]{hyperref}
\usepackage{bookmark}

\hypersetup{colorlinks=true,linkcolor=black,citecolor=black,filecolor=black,urlcolor=black,bookmarksdepth=section}

\makeatletter
\let\old@setaddresses\@setaddresses
\def\@setaddresses{\bigskip{\parindent 0pt\let\scshape\relax\let\ttfamily\relax\old@setaddresses}}
\makeatother

\setlist[enumerate]{label=\textup{(\arabic*)}, noitemsep, topsep=3pt plus 3pt, leftmargin=*}
\setlist[itemize]{label=\textbullet, noitemsep, topsep=3pt plus 3pt, labelsep=.5em, labelindent=.2em, leftmargin=*}
\newlist{enumeratei}{enumerate}{1}
\setlist[enumeratei]{label=\textup{(\roman*)}, leftmargin=*, widest=iii}

\newtheorem{theorem}{Theorem}[section]
\newtheorem{proposition}[theorem]{Proposition}
\newtheorem{lemma}[theorem]{Lemma}
\newtheorem{conjecture}[theorem]{Conjecture}
\theoremstyle{remark}
\newtheorem{claim}[theorem]{Claim}

\newcommand{\cB}{\mathcal{B}}
\newcommand{\cF}{\mathcal{F}}
\newcommand{\cH}{\mathcal{H}}
\newcommand{\cL}{\mathcal{L}}
\newcommand{\cR}{\mathcal{R}}

\newcommand{\R}{\mathbb{R}}

\let\leq\leqslant
\let\geq\geqslant
\let\setminus\smallsetminus
\let\int\undefined
\DeclareMathOperator{\int}{int}

\hypersetup{
  pdftitle={Colouring polygon visibility graphs and their generalizations},
  pdfauthor={James Davies, Tomasz Krawczyk, Rose McCarty, Bartosz Walczak}
}

\title{Colouring polygon visibility graphs and their generalizations}
\author{James Davies\and Tomasz Krawczyk\and Rose McCarty\and Bartosz Walczak}

\address[James Davies and Rose McCarty]{Department of Combinatorics and Optimization, School of Mathematics, University of Waterloo, Canada}
\email{\href{mailto:jgdavies@uwaterloo.ca}{jgdavies@uwaterloo.ca}, \href{mailto:rose.mccarty@uwaterloo.ca}{rose.mccarty@uwaterloo.ca}}

\address[Tomasz Krawczyk and Bartosz Walczak]{Department of Theoretical Computer Science, Faculty of Mathematics and Computer Science, Jagiellonian University, Kraków, Poland}
\email{\href{mailto:krawczyk@tcs.uj.edu.pl}{krawczyk@tcs.uj.edu.pl}, \href{mailto:walczak@tcs.uj.edu.pl}{walczak@tcs.uj.edu.pl}}

\thanks{Tomasz Krawczyk and Bartosz Walczak were partially supported by the National Science Centre of Poland grant 2015/17/D/ST1/00585.}

\begin{document}

\begin{abstract}
Curve pseudo-visibility graphs generalize polygon and pseudo-polygon visibility graphs and form a hereditary class of graphs.
We prove that every curve pseudo-visibility graph with clique number $\omega$ has chromatic number at most $3\cdot 4^{\omega-1}$.
The proof is carried through in the setting of ordered graphs; we identify two conditions satisfied by every curve pseudo-visibility graph (considered as an ordered graph) and prove that they are sufficient for the claimed bound.
The proof is algorithmic: both the clique number and a colouring with the claimed number of colours can be computed in polynomial time.
\end{abstract}

\maketitle

\section{Introduction}

A \emph{polygon} is a Jordan curve made of finitely many line segments.
A \emph{polygon visibility graph} is the graph on the set of vertices of a polygon $P$ that has an edge between each pair of \emph{mutually visible} vertices, which means that the line segment connecting them is disjoint from the exterior of $P$.
A class of graphs is \emph{$\chi$-bounded} if there is a function that bounds the chromatic number in terms of the clique number for every graph in the class.
A clique in a polygon visibility graph has a natural interpretation---it is the maximum size of a subset of the vertices whose convex hull is disjoint from the exterior of the polygon (see Figure~\ref{fig:graph_classes}, top-left).
The starting point of and main motivation for this work is the question of Kára, Pór, and Wood~\cite{KPW05} of whether the class of polygon visibility graphs is $\chi$-bounded.
We answer it in the affirmative.

\begin{theorem}
\label{thm:poly}
Every polygon visibility graph with clique number\/ $\omega$ has chromatic number at most\/ $3\cdot 4^{\omega-1}$.
\end{theorem}

The bound in Theorem~\ref{thm:poly} also holds for all induced subgraphs of polygon visibility graphs.
Such graphs can be defined alternatively as \emph{curve visibility graphs}, that is, visibility graphs of points on a Jordan curve, where two points are considered to be mutually visible if the line segment connecting them is disjoint from the exterior of the curve (see Figure~\ref{fig:graph_classes}, bottom-left).

\begin{figure}[t]
\centering
\begin{tikzpicture}[scale=0.19, every node/.style={inner sep=2, outer sep=0}, vtx/.style={draw, circle, fill=gray}]
  \coordinate (P1) at (-14,1) {};
  \coordinate (P2) at (-2,-11) {};
  \coordinate (P3) at (4,-7) {};
  \coordinate (P4) at (-4,-5) {};
  \coordinate (P5) at (4,3) {};
  \coordinate (P6) at (8.5,-4) {};
  \coordinate (P7) at (5,-11) {};
  \coordinate (P8) at (13,-11) {};
  \coordinate (P9) at (14,-5) {};
  \coordinate (P10) at (4,11) {};
  \coordinate (P11) at (-2,6) {};
  \coordinate (P12) at (-6,8) {};
  \useasboundingbox (-16,-16) rectangle (16,16);
  \draw[fill=gray!20] (P1)--(P12)--(P11)--(P5)--(P4)--cycle;
  \draw[line width=3pt, line join=round] (P1)--(P2)--(P3)--(P4)--(P5)--(P6)--(P7)--(P8)--(P9)--(P10)--(P11)--(P12)--cycle;
  \draw[red, thick] (P1)--(P2)--(P3)--(P4)--(P5)--(P6)--(P7)--(P8)--(P9)--(P10)--(P11)--(P12)--cycle;
  \draw[red, thick] (P1)--(P2);
  \draw[red, thick] (P1)--(P4);
  \draw[red, thick] (P1)--(P5);
  \draw[red, thick] (P1)--(P11);
  \draw[red, thick] (P1)--(P12);
  \draw[red, thick] (P2)--(P4);
  \draw[red, thick] (P4)--(P10);
  \draw[red, thick] (P4)--(P11);
  \draw[red, thick] (P4)--(P12);
  \draw[red, thick] (P5)--(P9);
  \draw[red, thick] (P5)--(P10);
  \draw[red, thick] (P6)--(P8);
  \draw[red, thick] (P6)--(P9);
  \draw[red, thick] (P6)--(P10);
  \draw[red, thick] (P7)--(P9);
  \draw[red, thick] (P8)--(P10);
  \node[vtx] at (P1) {};
  \node[vtx] at (P2) {};
  \node[vtx] at (P3) {};
  \node[vtx] at (P4) {};
  \node[vtx] at (P6) {};
  \node[vtx] at (P7) {};
  \node[vtx] at (P8) {};
  \node[vtx] at (P9) {};
  \node[vtx] at (P10) {};
  \node[vtx] at (P11) {};
  \draw[red, thick] (P5)--(P12);
  \node[vtx] at (P5) {};
  \node[vtx] at (P12) {};
  \node at ($(P1)+(0,-0.5)$) [below] {$1$};
  \node at ($(P2)+(0,-0.5)$) [below] {$2$};
  \node at ($(P3)+(0,0.5)$) [above] {$3$};
  \node at ($(P4)+(1.6,0.7)$) [right] {$4$};
  \node at ($(P5)+(-0.2,-1)$) [below] {$5$};
  \node at ($(P6)+(-0.7,0)$) [left] {$6$};
  \node at ($(P7)+(0,-0.5)$) [below] {$7$};
  \node at ($(P8)+(0,-0.5)$) [below] {$8$};
  \node at ($(P9)+(0.5,0)$) [right] {$9$};
  \node at ($(P10)+(0,0.5)$) [above] {$10$};
  \node at ($(P11)+(0,1)$) [above] {$11$};
  \node at ($(P12)+(0,0.5)$) [above] {$12$};
\end{tikzpicture}
\hspace{0.6cm}
\begin{tikzpicture}[scale=0.19, every node/.style={inner sep=2, outer sep=0}, vtx/.style={draw, circle, fill=gray}]
  \coordinate (A1) at (-2,8.5) {};
  \coordinate (A2) at (0,0) {};
  \coordinate (A3) at (-9,-2) {};
  \coordinate (A4) at (2,-9) {};
  \coordinate (A5) at (9,3) {};
  \coordinate (L124_B1_1) at (-2.5,14) {};
  \coordinate (L124_A1_1) at (-2.2,6) {};
  \coordinate (L124_A1_2) at (-2,4) {};
  \coordinate (L124_A2_1) at (3,-2) {};
  \coordinate (L124_A2_2) at (3,-6) {};
  \coordinate (L124_A4_1) at (1,-14) {};
  \coordinate (L13_B1_1) at (8,14) {};
  \coordinate (L13_B1_2) at (4,11) {};
  \coordinate (L13_A3_1) at (-9,-8) {};
  \coordinate (L13_A3_2) at (-10,-14) {};
  \coordinate (L325_B3_1) at (-14,-2) {};
  \coordinate (L325_A3_1) at (-7,-2.3) {};
  \coordinate (L325_A3_2) at (-4,-2) {};
  \coordinate (L325_A2_1) at (2,2) {};
  \coordinate (L325_A2_2) at (6,3.2) {};
  \coordinate (L325_A5_1) at (14,3) {};
  \coordinate (L15_B1_1) at (-14,6) {};
  \coordinate (L15_B1_2) at (-8,6.75) {};
  \coordinate (L15_A1_1) at (2,8) {};
  \coordinate (L15_A1_2) at (6,6) {};
  \coordinate (L15_A5_1) at (11,-2) {};
  \coordinate (L15_A5_2) at (14,-6) {};
  \coordinate (L34_B3_1) at (-10,14) {};
  \coordinate (L34_B3_2) at (-10,8) {};
  \coordinate (L34_A3_1) at (-7,-4) {};
  \coordinate (L34_A3_2) at (-4,-6) {};
  \coordinate (L34_A4_1) at (14,-12) {};
  \coordinate (L45_B4_1) at (-14,-11) {};
  \coordinate (L45_A4_1) at (5,-8) {};
  \coordinate (L45_A4_2) at (9,-3.5) {};
  \coordinate (L45_A4_3) at (9.25,0) {};
  \coordinate (L45_A5_1) at (14,10) {};
  \useasboundingbox (-17,-16) rectangle (17,16);
  \draw[use Hobby shortcut] (L124_B1_1) .. (A1) .. (L124_A1_1) .. (L124_A1_2) .. (A2) .. (L124_A2_1) .. (L124_A2_2) .. (A4) .. (L124_A4_1);
  \draw[use Hobby shortcut] (L325_B3_1) .. (A3) .. (L325_A3_1) .. (L325_A3_2) .. (A2) .. (L325_A2_1) .. (L325_A2_2) .. (A5) .. (L325_A5_1);
  \draw[use Hobby shortcut] (L15_B1_1) .. (L15_B1_2) .. (A1) .. (L15_A1_1) .. (L15_A1_2) .. (A5) .. (L15_A5_1) .. (L15_A5_2);
  \draw[use Hobby shortcut] (L13_B1_1) .. (L13_B1_2) .. (A1) .. (A3) .. (L13_A3_1) .. (L13_A3_2);
  \draw[use Hobby shortcut] (L34_B3_1) .. (L34_B3_2) .. (A3) .. (L34_A3_1) .. (L34_A3_2) .. (A4) .. (L34_A4_1);
  \draw[use Hobby shortcut] (L45_B4_1) .. (A4) .. (L45_A4_1) .. (L45_A4_2) .. (L45_A4_3) .. (A5) .. (L45_A5_1);
  \draw[line width=3pt, use Hobby shortcut] (L124_B1_1) .. ([blank]A1) .. (L124_A1_1) .. (L124_A1_2) .. (A2) .. ([blank]L124_A2_1) .. ([blank]L124_A2_2) .. ([blank]A4) .. ([blank]L124_A4_1);
  \draw[line width=3pt, use Hobby shortcut] (L325_B3_1) .. ([blank]A3) .. (L325_A3_1) .. (L325_A3_2) .. (A2) .. ([blank]L325_A2_1) .. ([blank]L325_A2_2) .. ([blank]A5) .. ([blank]L325_A5_1);
  \draw[line width=3pt, use Hobby shortcut] (L15_B1_1) .. ([blank]L15_B1_2) .. ([blank]A1) .. (L15_A1_1) .. (L15_A1_2) .. (A5) .. ([blank]L15_A5_1) .. ([blank]L15_A5_2);
  \draw[line width=3pt, use Hobby shortcut] (L34_B3_1) .. ([blank]L34_B3_2) .. ([blank]A3) .. (L34_A3_1) .. (L34_A3_2) .. (A4) .. ([blank]L34_A4_1);
  \draw[line width=3pt, use Hobby shortcut] (L45_B4_1) .. ([blank]A4) .. (L45_A4_1) .. (L45_A4_2) .. (L45_A4_3) .. (A5) .. ([blank]L45_A5_1);
  \node[vtx] at (A2) {};
  \draw[red, thick, use Hobby shortcut] (L124_B1_1) .. ([blank]A1) .. (L124_A1_1) .. (L124_A1_2) ..(A2) .. (L124_A2_1) ..(L124_A2_2) .. (A4) .. ([blank]L124_A4_1);
  \draw[red, thick, use Hobby shortcut] (L325_B3_1) .. ([blank]A3) .. (L325_A3_1) .. (L325_A3_2) .. (A2) .. (L325_A2_1) .. (L325_A2_2) .. (A5) .. ([blank]L325_A5_1);
  \draw[red, thick, use Hobby shortcut] (L15_B1_1) .. ([blank]L15_B1_2) .. ([blank]A1) .. (L15_A1_1) .. (L15_A1_2) .. (A5) .. ([blank]L15_A5_1) .. ([blank]L15_A5_2);
  \draw[red, thick, use Hobby shortcut] (L34_B3_1) .. ([blank]L34_B3_2) .. ([blank]A3) .. (L34_A3_1) .. (L34_A3_2) .. (A4) .. ([blank]L34_A4_1);
  \draw[red, thick, use Hobby shortcut] (L45_B4_1) .. ([blank]A4) .. (L45_A4_1) .. (L45_A4_2) .. (L45_A4_3) .. (A5) .. ([blank]L45_A5_1);
  \node[vtx] at (A1) {};
  \node[vtx] at (A3) {};
  \node[vtx] at (A4) {};
  \node[vtx] at (A5) {};
  \node at ($(A1)+(-1,0.5)$) [above] {$1$};
  \node at ($(A2)+(0,-0.65)$) [below] {$2$};
  \node at ($(A3)+(-1,-0.5)$) [below] {$3$};
  \node at ($(A4)+(0.8,-0.6)$) [below] {$4$};
  \node at ($(A5)+(1.2,0.35)$) [above] {$5$};
  \node at (L34_B3_1) [above] {$34$};
  \node at (L34_A4_1) [right] {$34$};
  \node at (L124_B1_1) [above] {$124$};
  \node at (L124_A4_1) [below] {$124$};
  \node at (L325_B3_1) [left] {$325$};
  \node at (L325_A5_1) [right] {$325$};
  \node at (L15_B1_1) [left] {$15$};
  \node at (L15_A5_2) [right] {$15$};
  \node at (L13_B1_1) [above] {$13$};
  \node at (L13_A3_2) [below] {$13$};
  \node at (L45_B4_1) [left] {$45$};
  \node at (L45_A5_1) [right] {$45$};
\end{tikzpicture}\\[0.6cm]
\begin{tikzpicture}[scale=0.19, every node/.style={inner sep=2, outer sep=0}, vtx/.style={draw, circle, fill=gray}]
  \coordinate (P1) at (-14,1) {};
  \coordinate (P2) at (-2,-11) {};
  \coordinate (P3) at (4,-7) {};
  \coordinate (P4) at (-4,-5) {};
  \coordinate (P4P5_1) at (-4,-4) {};
  \coordinate (P4P5_2) at (0,1.5) {};
  \coordinate (P5) at (4,3) {};
  \coordinate (P6) at (8.5,-4) {};
  \coordinate (P7) at (5,-11) {};
  \coordinate (P8) at (13,-11) {};
  \coordinate (P9) at (14,-5) {};
  \coordinate (P10) at (4,11) {};
  \coordinate (P11) at (-2,6) {};
  \coordinate (P11P12) at (-4,6.5) {};
  \coordinate (P12) at (-6,8) {};
  \useasboundingbox (-16,-16) rectangle (16,16);
  \draw[line width=3pt, use Hobby shortcut] ([closed]P1)..(P2)..(P3)..(P4)..(P4P5_1)..(P4P5_2)..(P5)..(P6)..(P7)..(P8)..(P9)..(P10)..(P11)..(P11P12)..(P12);
  \draw[red, thick] (P1)--(P2);
  \draw[red, thick] (P1)--(P4);
  \draw[red, thick] (P1)--(P5);
  \draw[red, thick] (P1)--(P11);
  \draw[red, thick] (P1)--(P12);
  \draw[red, thick] (P2)--(P3);
  \draw[red, thick] (P2)--(P4);
  \draw[red, thick] (P4)--(P12);
  \draw[red, thick] (P5)--(P10);
  \draw[red, thick] (P5)--(P11);
  \draw[red, thick] (P6)--(P7);
  \draw[red, thick] (P6)--(P8);
  \draw[red, thick] (P6)--(P9);
  \draw[red, thick] (P7)--(P8);
  \draw[red, thick] (P7)--(P9);
  \draw[red, thick] (P8)--(P9);
  \draw[red, thick] (P8)--(P10);
  \draw[red, thick] (P9)--(P10);
  \draw[red, thick] (P9)--(P10);
  \node[vtx] at (P1) {};
  \node[vtx] at (P2) {};
  \node[vtx] at (P3) {};
  \node[vtx] at (P4) {};
  \node[vtx] at (P5) {};
  \node[vtx] at (P6) {};
  \node[vtx] at (P7) {};
  \node[vtx] at (P8) {};
  \node[vtx] at (P9) {};
  \node[vtx] at (P10) {};
  \node[vtx] at (P11) {};
  \node[vtx] at (P12) {};
  \node at ($(P1)+(-0.8,-0.5)$) [below] {$1$};
  \node at ($(P2)+(0,-0.5)$) [below] {$2$};
  \node at ($(P3)+(0.3,0.5)$) [above] {$3$};
  \node at ($(P4)+(1.6,0.9)$) [right] {$4$};
  \node at ($(P5)+(-0.2,-1)$) [below] {$5$};
  \node at ($(P6)+(-0.7,0)$) [left] {$6$};
  \node at ($(P7)+(-0.8,-0.5)$) [below] {$7$};
  \node at ($(P8)+(0.8,-0.5)$) [below] {$8$};
  \node at ($(P9)+(0.5,0)$) [right] {$9$};
  \node at ($(P10)+(-0.2,0.5)$) [above] {$10$};
  \node at ($(P11)+(0,1)$) [above] {$11$};
  \node at ($(P12)+(0,0.5)$) [above] {$12$};
\end{tikzpicture}
\hspace{0.6cm}
\begin{tikzpicture}[scale=0.19, every node/.style={inner sep=2, outer sep=0}, vtx/.style={draw, circle, fill=gray}]
  \coordinate (A1) at (-2,8.5) {};
  \coordinate (A2) at (0,0) {};
  \coordinate (A3) at (-9,-2) {};
  \coordinate (A4) at (2,-9) {};
  \coordinate (A5) at (9,3) {};
  \coordinate (IA1A2_1) at (-4.3,5.2) {};
  \coordinate (IA1A2_2) at (-2.5,1) {};
  \coordinate (IA2A3_1) at (1.25,-2.5) {};
  \coordinate (IA2A3_2) at (-4,-4) {};
  \coordinate (IA3A4_1) at (-11,-1) {};
  \coordinate (IA3A4_2) at (-12.5,-4.5) {};
  \coordinate (IA3A4_3) at (-9,-9.5) {};
  \coordinate (IA3A4_4) at (0,-10) {};
  \coordinate (IA4A5_1) at (6,-9.5) {};
  \coordinate (IA4A5_2) at (11,-8) {};
  \coordinate (IA4A5_3) at (12,0) {};
  \coordinate (IA5A1_1) at (8,8) {};
  \coordinate (IA5A1_2) at (5,11) {};
  \coordinate (IA5A1_3) at (1,11) {};
  \coordinate (L124_B1_1) at (-2.5,14) {};
  \coordinate (L124_A1_1) at (-2.2,6) {};
  \coordinate (L124_A1_2) at (-2,4) {};
  \coordinate (L124_A2_1) at (3,-2) {};
  \coordinate (L124_A2_2) at (3,-6) {};
  \coordinate (L124_A4_1) at (1,-14) {};
  \coordinate (L13_B1_1) at (8,14) {};
  \coordinate (L13_B1_2) at (4,11) {};
  \coordinate (L13_A3_1) at (-9,-8) {};
  \coordinate (L13_A3_2) at (-10,-14) {};
  \coordinate (L325_B3_1) at (-14,-2) {};
  \coordinate (L325_A3_1) at (-7,-2.3) {};
  \coordinate (L325_A3_2) at (-4,-2) {};
  \coordinate (L325_A2_1) at (2,2) {};
  \coordinate (L325_A2_2) at (6,3.2) {};
  \coordinate (L325_A5_1) at (14,3) {};
  \coordinate (L15_B1_1) at (-14,6) {};
  \coordinate (L15_B1_2) at (-8,6.75) {};
  \coordinate (L15_A1_1) at (2,8) {};
  \coordinate (L15_A1_2) at (6,6) {};
  \coordinate (L15_A5_1) at (11,-2) {};
  \coordinate (L15_A5_2) at (14,-6) {};
  \coordinate (L34_B3_1) at (-10,14) {};
  \coordinate (L34_B3_2) at (-10,8) {};
  \coordinate (L34_A3_1) at (-7,-4) {};
  \coordinate (L34_A3_2) at (-4,-6) {};
  \coordinate (L34_A4_1) at (14,-12) {};
  \coordinate (L45_B4_1) at (-14,-11) {};
  \coordinate (L45_A4_1) at (5,-8) {};
  \coordinate (L45_A4_2) at (9,-3.5) {};
  \coordinate (L45_A4_3) at (9.25,0) {};
  \coordinate (L45_A5_1) at (14,10) {};
  \useasboundingbox (-17,-16) rectangle (17,16);
  \draw[line width=3pt, use Hobby shortcut] ([closed]A1) .. (IA1A2_1) .. (IA1A2_2) .. (A2) .. (IA2A3_1) .. (IA2A3_2) .. (A3) .. (IA3A4_1) .. (IA3A4_2) .. (IA3A4_3) .. (IA3A4_4) .. (A4) .. (IA4A5_1) .. (IA4A5_2) .. (IA4A5_3) .. (A5) .. (IA5A1_1) .. (IA5A1_2) .. (IA5A1_3);
  \draw[use Hobby shortcut] (L124_B1_1) .. (A1) .. (L124_A1_1) .. (L124_A1_2) .. (A2) .. (L124_A2_1) .. (L124_A2_2) .. (A4) .. (L124_A4_1);
  \draw[use Hobby shortcut] (L325_B3_1) .. (A3) .. (L325_A3_1) .. (L325_A3_2) .. (A2) .. (L325_A2_1) .. (L325_A2_2) .. (A5) .. (L325_A5_1);
  \draw[use Hobby shortcut] (L15_B1_1) .. (L15_B1_2) .. (A1) .. (L15_A1_1) .. (L15_A1_2) .. (A5) .. (L15_A5_1) .. (L15_A5_2);
  \draw[use Hobby shortcut] (L13_B1_1) .. (L13_B1_2) .. (A1) .. (A3) .. (L13_A3_1) .. (L13_A3_2);
  \draw[use Hobby shortcut] (L34_B3_1) .. (L34_B3_2) .. (A3) .. (L34_A3_1) .. (L34_A3_2) .. (A4) .. (L34_A4_1);
  \draw[use Hobby shortcut] (L45_B4_1) .. (A4) .. (L45_A4_1) .. (L45_A4_2) .. (L45_A4_3) .. (A5) .. (L45_A5_1);
  \draw[red, thick, use Hobby shortcut] (L325_B3_1) .. ([blank]A3) .. ([blank]L325_A3_1) .. ([blank]L325_A3_2) .. ([blank]A2) .. (L325_A2_1) ..(L325_A2_2) .. (A5) .. ([blank]L325_A5_1);
  \draw[red, thick, use Hobby shortcut] (L15_B1_1) .. ([blank]L15_B1_2) .. ([blank]A1) .. (L15_A1_1) .. (L15_A1_2) .. (A5) .. ([blank]L15_A5_1) .. ([blank]L15_A5_2);
  \draw[red, thick, use Hobby shortcut] (L34_B3_1) .. ([blank]L34_B3_2) .. ([blank]A3) .. (L34_A3_1) .. (L34_A3_2) .. (A4) .. ([blank]L34_A4_1);
  \draw[red, thick, use Hobby shortcut] (L45_B4_1) .. ([blank]A4) .. (L45_A4_1) .. (L45_A4_2) .. (L45_A4_3) .. (A5) .. ([blank]L45_A5_1);
  \node[vtx] at (A2) {};
  \node[vtx] at (A3) {};
  \node[vtx] at (A5) {};
  \draw[red, thick, use Hobby shortcut] (L124_B1_1) .. ([blank]A1) .. (L124_A1_1) .. (L124_A1_2) ..(A2) .. (L124_A2_1) .. (L124_A2_2) .. (A4) .. ([blank]L124_A4_1);
  \node[vtx] at (A1) {};
  \node[vtx] at (A4) {};
  \node at ($(A1)+(-1,0.5)$) [above] {$1$};
  \node at ($(A2)+(0,-0.65)$) [below] {$2$};
  \node at ($(A3)+(-1,-0.5)$) [below] {$3$};
  \node at ($(A4)+(0.8,-0.6)$) [below] {$4$};
  \node at ($(A5)+(1.2,0.35)$) [above] {$5$};
  \node at (L34_B3_1) [above] {$34$};
  \node at (L34_A4_1) [right] {$34$};
  \node at (L124_B1_1) [above] {$124$};
  \node at (L124_A4_1) [below] {$124$};
  \node at (L325_B3_1) [left] {$325$};
  \node at (L325_A5_1) [right] {$325$};
  \node at (L15_B1_1) [left] {$15$};
  \node at (L15_A5_2) [right] {$15$};
  \node at (L13_B1_1) [above] {$13$};
  \node at (L13_A3_2) [below] {$13$};
  \node at (L45_B4_1) [left] {$45$};
  \node at (L45_A5_1) [right] {$45$};
\end{tikzpicture}
\caption{From left to right: a polygon visibility graph (where the convex hull of a maximum clique is shaded), a pseudo-polygon visibility graph, a curve visibility graph, and a curve pseudo-visibility graph.
A ``visibility'' between each pair of adjacent vertices is drawn with a red (pseudo-)segment.}
\label{fig:graph_classes}
\end{figure}
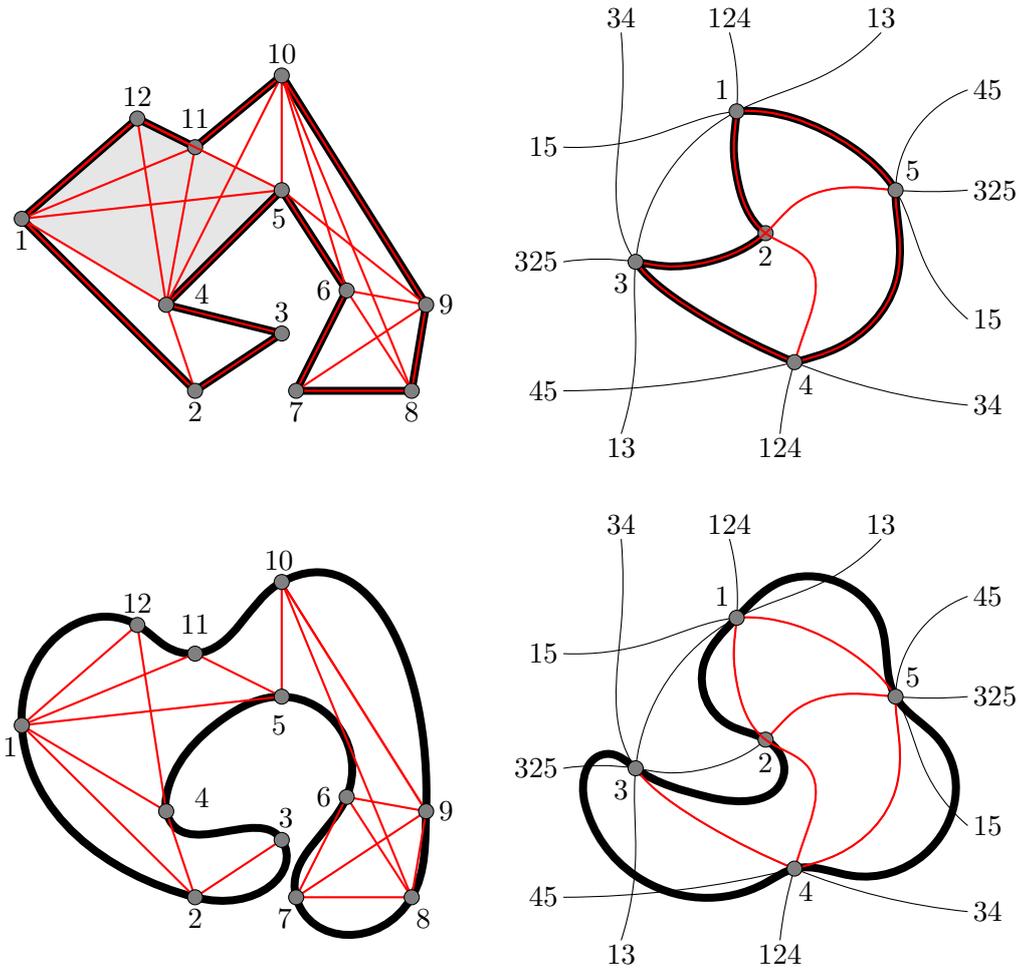

O'Rourke and Streinu~\cite{ORS97} studied visibility graphs of pseudo-polygons (polygons on pseudoline arrangements; see Figure~\ref{fig:graph_classes}, top-right), where two vertices of the polygon are considered to be mutually visible if the pseudoline segment connecting them in the arrangement is disjoint from the exterior of the polygon.
As a common generalization of these graphs and curve visibility graphs, we define curve pseudo-visibility graphs as follows.
For a pseudoline arrangement $\cL$, a Jordan curve $K$, and a finite set $V$ of points on $K$ any two of which lie on a common pseudoline in $\cL$, the \emph{curve pseudo-visibility graph} $G_\cL(K,V)$ has vertex set $V$ and has an edge between each pair of vertices such that the pseudoline segment in $\cL$ connecting them is disjoint from the exterior of $K$ (see Figure~\ref{fig:graph_classes}, bottom-right).
We elaborate on this notion in Section~\ref{sec:pseudo}; in particular, we show that curve pseudo-visibility graphs are exactly the induced subgraphs of the visibility graphs of pseudo-polygons.
With this notion in hand, we provide the following topological generalization of Theorem~\ref{thm:poly}.

\begin{theorem}
\label{thm:pseudo}
Every curve pseudo-visibility graph with clique number\/ $\omega$ has chromatic number at most\/ $3\cdot 4^{\omega-1}$.
\end{theorem}

To prove Theorem~\ref{thm:pseudo} (and thus Theorem~\ref{thm:poly}), we turn our attention to ordered graphs, where an \emph{ordered graph} is a pair $(G,{\prec})$ such that $G$ is a graph and $\prec$ is a linear order on the vertices of $G$.
A curve pseudo-visibility graph comes with a natural linear order on the vertices (determined up to rotation), which makes it an ordered graph; it is the order in which the vertices are encountered when following the Jordan curve in the counterclockwise direction starting from an arbitrarily chosen vertex.
An ordered graph $(H,{\prec_H})$ is an (\emph{induced}) \emph{ordered subgraph} of an ordered graph $(G,{\prec})$ if $H$ is a subgraph (an induced subgraph, respectively) of $G$ and $\prec_H$ is the restriction of $\prec$ to the vertices of $H$.
In Section~\ref{sec:obstructions}, we provide two natural families of ordered obstructions to (that is, ordered graphs that cannot occur as induced ordered subgraphs of) curve pseudo-visibility graphs: the family $\cH$ that we define in Section~\ref{sec:obstructions} and the family of ordered holes (see Figure~\ref{fig:obstructions}), both easily verifiable in polynomial time.
We prove the following further generalization of Theorem~\ref{thm:pseudo}.

\begin{figure}[t]
\centering
\begin{tikzpicture}[scale=.8, every node/.style={inner sep=2, outer sep=0, draw, circle, fill=gray}]
  \draw[ultra thin] (0,0) circle (2) {};
  \node[label=left:$u$] (u) at (180:2) {};
  \node[label=right:$v$] (v) at (0:2) {};
  \foreach\i in {1,...,4}{
    \node[inner sep=1.5] (A\i) at (170+\i*36:2) {};
    \node[inner sep=1.5] (B\i) at (190+\i*36:2) {};
  }
  \draw[very thick] (u) -- (B1);
  \draw[very thick] (A1) -- (B2);
  \draw[very thick] (A2) -- (B3);
  \draw[very thick] (A3) -- (B4);
  \draw[very thick] (A4) -- (v);
  \node[inner sep=1.5] (C1) at (50:2) {};
  \node[inner sep=1.5] (C2) at (130:2) {};
  \node[inner sep=1.5] (D) at (90:2) {};
  \draw[very thick] (v) -- (D);
  \draw[very thick] (C1) -- (C2);
  \draw[very thick] (D) -- (u);
  \draw[very thick, dashed] (u) -- (v);
\end{tikzpicture}\hskip 1.5cm
\begin{tikzpicture}[scale=.8, every node/.style={inner sep=2, outer sep=0, draw, circle, fill=gray}]
  \draw[ultra thin] (0,0) circle (2) {};
  \foreach\i in {1,...,5}{\node (A\i) at (\i*72:2) {};}
  \draw[very thick] (A1) -- (A2) -- (A3) -- (A4) -- (A5) -- (A1);
  \draw[very thick, dashed] (A1) -- (A3) -- (A5) -- (A2) -- (A4) -- (A1);
\end{tikzpicture}\hskip 1.5cm
\begin{tikzpicture}[scale=.6, every node/.style={inner sep=2, outer sep=0, draw, circle, fill=gray}]
  \node[label=below:$\vphantom{b}a$] (a) at (0,0) {};
  \node[label=below:$b$] (b) at (2,0) {};
  \node[label=below:$\vphantom{b}c$] (c) at (4,0) {};
  \node[label=below:$d$] (d) at (6,0) {};
  \draw[thick] (a) to [bend left=50] (c);
  \draw[thick] (b) to [bend left=50] (d);
  \draw[thick, dashed] (a) to [bend left=50] (d);
\end{tikzpicture}
\caption{A graph in $\cH$ (left), an ordered hole (middle), and the forbidden configuration for a capped graph (right).
Dashed lines indicate non-edges.
The pairs of vertices where no lines are drawn can be edges or non-edges.}
\label{fig:obstructions}
\end{figure}
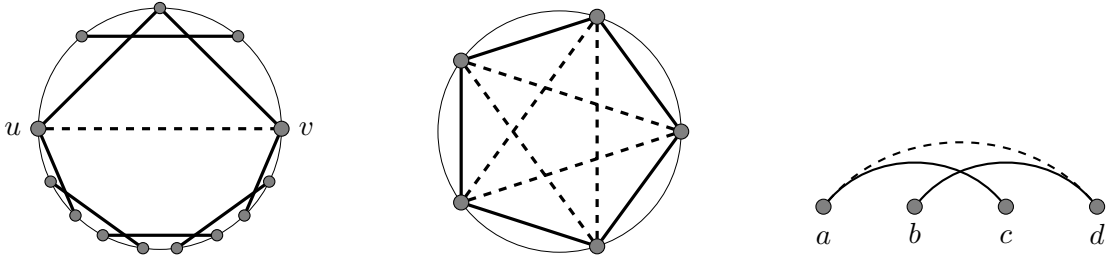

\begin{theorem}
\label{thm:H-free}
Every\/ $\cH$-free ordered graph with clique number\/ $\omega\geq 2$ has chromatic number at most\/ $3\cdot 4^\omega(\omega-1)$ in general and at most\/ $3\cdot 4^{\omega-1}$ when also ordered-hole-free.
Moreover, there is a polynomial-time algorithm that takes in an\/ $\cH$-free ordered graph and computes its clique number\/ $\omega$ and a colouring with the claimed number of colours.
\end{theorem}

Our proofs of Theorems~\ref{thm:poly}--\ref{thm:H-free} ultimately lead to the class of capped graphs, which may be of independent interest.
A \emph{capped graph} is an ordered graph $(G,{\prec})$ such that for any four vertices $a\prec b\prec c\prec d$, if $ac,bd\in E(G)$, then $ad\in E(G)$; see Figure~\ref{fig:obstructions} (right).
We note that this condition has been previously studied for terrain visibility graphs~\cite{Abello95terrain,Ameer2020}, where it is called the ``$X$-property''.
In Section~\ref{sec:partition}, we show that the vertices of any $\cH$-free ordered graph can be partitioned into three sets each inducing a capped graph.
This way, Theorem~\ref{thm:H-free} becomes a corollary to the following.

\begin{theorem}
\label{thm:capped}
Every capped graph with clique number\/ $\omega\geq 2$ has chromatic number at most\/ $4^\omega(\omega-1)$ in general and at most\/ $4^{\omega-1}$ when also ordered-hole-free.
Moreover, there is a polynomial-time algorithm that takes in a capped graph and computes its clique number\/ $\omega$ and a colouring with the claimed number of colours.
\end{theorem}

We prove Theorem~\ref{thm:capped} in Section~\ref{sec:capped}.
Any improvement on the bounds in Theorem~\ref{thm:capped} would immediately imply corresponding improvements in Theorems~\ref{thm:poly}--\ref{thm:H-free}.
A major open problem for most known $\chi$-bounded classes of graphs is whether they are polynomially $\chi$-bounded, that is, whether the chromatic number of the graphs in the class is bounded by a polynomial function of their clique number.
Esperet~\cite{esperet2017habilitation} conjectured that \emph{every} hereditary class of graphs that is $\chi$-bounded is polynomially $\chi$-bounded.
While we have little faith in this conjecture, we do expect that it holds for capped graphs (and, consequently, for the graphs considered in Theorems~\ref{thm:poly}--\ref{thm:H-free}).

\begin{conjecture}
There is a polynomial function\/ $p$ such that every capped graph with clique number\/ $\omega$ has chromatic number at most\/ $p(\omega)$.
\end{conjecture}

While our proof of Theorem~\ref{thm:capped} is direct, we remark that a recent result of Scott and Seymour~\cite{banana20} implies $\chi$-boundedness (with a much weaker bound) of the significantly broader class of $X$-free ordered graphs, that is, ordered graphs excluding the four-vertex ordered graph $X$ illustrated in Figure~\ref{fig:banana_X} (right) as an induced ordered subgraph.
In particular, every capped graph is $X$-free.
Tomon~\cite{Tomon-personal} conjectured that the class of $X$-free ordered graphs is $\chi$-bounded.
This statement implies not only Theorem~\ref{thm:capped} but also the theorem of Rok and Walczak~\cite{RWouterstring} that so-called outerstring graphs are $\chi$-bounded.
This is because outerstring graphs (with the natural linear order on the vertices) are easily seen to be $X$-free.
Scott and Seymour~\cite{banana20} proved that for every graph $H$ that is a ``banana'' (or more generally---a ``banana tree''), the class of graphs excluding all subdivisions of $H$ as induced subgraphs is $\chi$-bounded.
Figure~\ref{fig:banana_X} (left) shows an example of a ``banana'' $B_4$ with the property that no subdivision of $B_4$ can be made $X$-free under any order of the vertices.
This shows that the aforementioned result of Scott and Seymour implies Tomon's conjecture.
We present more details in Section~\ref{sec:bananas}.

\begin{figure}[t]
\centering
\begin{tikzpicture}[scale=.5, every node/.style={inner sep=2, outer sep=0, draw, circle, fill=gray}]
  \node[label=left:$s$] (s) at (-4,0) {};
  \node[label=right:$t$] (t) at (4,0) {};
  \node[label=above:$x_1$] (X1) at (-2,2.5) {};
  \node[label=above:$x_2$] (X2) at (0,2.5) {};
  \node[label=above:$x_3$] (X3) at (2,2.5) {};
  \node[label=above:$y_1$] (Y1) at (-2,0) {};
  \node[label=above:$y_2$] (Y2) at (0,0) {};
  \node[label=above:$y_3$] (Y3) at (2,0) {};
  \node[label=above:$z_1$] (Z1) at (-2,-2.5) {};
  \node[label=above:$z_2$] (Z2) at (0,-2.5) {};
  \node[label=above:$z_3$] (Z3) at (2,-2.5) {};
  \draw[very thick] (s) -- (X1) -- (X2) -- (X3) -- (t);
  \draw[very thick] (s) -- (Y1) -- (Y2) -- (Y3) -- (t);
  \draw[very thick] (s) -- (Z1) -- (Z2) -- (Z3) -- (t);
\end{tikzpicture}\hskip 1.5cm
\begin{tikzpicture}[scale=.7, every node/.style={inner sep=2, outer sep=0, draw, circle, fill=gray}]
  \node[label=below:$\vphantom{b}a$] (a) at (0,0) {};
  \node[label=below:$b$] (b) at (2,0) {};
  \node[label=below:$\vphantom{b}c$] (c) at (4,0) {};
  \node[label=below:$d$] (d) at (6,0) {};
  \draw[thick] (a) to [bend left=50] (c);
  \draw[thick] (b) to [bend left=50] (d);
  \draw[thick, dashed] (a) to [bend left=60] (d);
  \draw[thick, dashed] (a) to [bend left=40] (b);
  \draw[thick, dashed] (b) to [bend left=40] (c);
  \draw[thick, dashed] (c) to [bend left=40] (d);
\end{tikzpicture}
\caption{The banana $B_4$ (left) and the ordered graph $X$ (right).}
\label{fig:banana_X}
\end{figure}
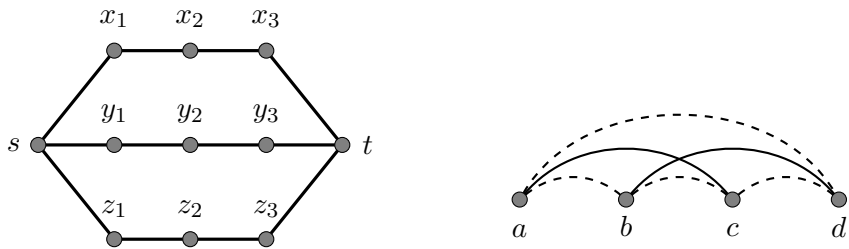

\begin{theorem}
\label{thm:X-free}
The class of\/ $X$-free ordered graphs is\/ $\chi$-bounded.
\end{theorem}

We conclude the introduction with a brief literature review in order to place Theorems \ref{thm:poly}--\ref{thm:capped} and~\ref{thm:X-free} in context.

\subsection*{Geometric graph classes and $\chi$-boundedness}

Various classic examples of $\chi$-bounded graph classes are defined in terms of geometric representations.
For instance, intersection graphs of axis-parallel rectangles~\cite{Asplud1960rectangle} and circle graphs~\cite{gyarfas1985chromatic} are $\chi$-bounded.
Most of the literature in this direction focuses on intersection or disjointness graphs of objects in the plane.
While the class of intersection graphs of curves in the plane is not itself $\chi$-bounded~\cite{pawlik2014triangle}, some very general subclasses are~\cite{chudnovsky2016induced,RokWalczak19}.
There are also very precise results for disjointness graphs of certain kinds of curves in the plane~\cite{PachTomon20}.

Less is known about $\chi$-boundedness of visibility graphs, even though various kinds of such graphs have been considered in the literature---see~\cite{survey13} for a survey.
Kára, Pór, and Wood~\cite{KPW05} conjectured that the class of point visibility graphs is $\chi$-bounded, but this was disproven by Pfender~\cite{Pfender08}.
Some types of bar visibility graphs are related to interval graphs~\cite{barKVis} and planar graphs~\cite{LMW87} and are therefore known to be $\chi$-bounded.

Axenovich, Rollin, and Ueckerdt~\cite{axenovich2018chromatic} considered the problem of whether ordered graphs excluding a fixed ordered graph $(H,{\prec})$ as an ordered subgraph (not necessarily induced) have bounded chromatic number; they showed various cases of $(H,{\prec})$ for which the answers are positive and negative.
In particular, the answer is negative if $H$ contains a cycle (as it is for unordered graphs), but they showed it is also negative for some acyclic ordered graphs $(H,{\prec})$.
Pach and Tomon~\cite{PachTomon20} used some specific classes of forbidden induced ordered graphs as a tool for studying $\chi$-boundedness of disjointness graphs of curves.
Max point-tolerance graphs~\cite{maxTolerance} and classes of graphs of bounded twin-width~\cite{twinWidthIII} are also known to be $\chi$-bounded and have well-understood characterizations as ordered graphs.

\subsection*{Algorithmic considerations}

The class of curve pseudo-visibility graphs is hereditary, whereas most well-known classes of visibility graphs are not, including the classes of point visibility graphs, polygon visibility graphs, and pseudo-visibility graphs.
The condition of the class being hereditary is very natural to impose when studying $\chi$-boundedness and implies that curve pseudo-visibility graphs can be characterized by excluded induced (ordered) subgraphs.
There has been a good deal of work on the characterization and recognition problems, but for point visibility graphs and polygon visibility graphs the problems appear to be hard (see~\cite{CHReals17} and~\cite{Ghosh97}).

These difficult characterization problems tend to become tractable, and have more natural solutions, in the ``pseudo-visibility setting''~\cite{Abello95terrain,AKMatroid02,Evans2015,ORS97}.
This is due to the connection between stretchability of pseudoline arrangements and representability of rank-$3$ oriented matroids.
So the pseudo-visibility setting is more combinatorial because it suffices to find the associated rank-$3$ oriented matroid without worrying about representability.
Representability provides a real difficulty; the pseudo-visibility setting is strictly more general for polygon visibility graphs~\cite{Streinu05}, even when certain restrictions are imposed~\cite{Ameer2020}.

It is an interesting problem to characterize ordered curve pseudo-visibility graphs by excluded induced ordered subgraphs.
The two aforementioned classes of obstructions ($\cH$ and the ordered holes) are likely to be insufficient for a full characterization---they roughly correspond to the first two of the four necessary conditions for an ordered graph to be a polygon visibility graph described by Ghosh~\cite{Ghosh97}.
Nevertheless, we conjecture the following.

\begin{conjecture}
Ordered curve pseudo-visibility graphs can be recognized in polynomial time.
\end{conjecture}

The part of Theorem~\ref{thm:H-free} concerning polynomial-time computation of clique number extends well-known results regarding polygon visibility graphs~\cite{clique85, clique00, algCliqueVis07}, although our algorithm is certainly slower.
We cannot expect to get an exact algorithm for the chromatic number, as Çağırıcı, Hliněný, and Roy~\cite{CHR19} proved that it is NP-complete to decide if a polygon visibility graph is $5$-colourable, even when the polygon is provided as part of the input.

\section{Curve pseudo-visibility graphs}
\label{sec:pseudo}

A \emph{pseudoline} is a simple curve which separates the plane into two unbounded regions.
A \emph{pseudoline arrangement} is a set of pseudolines such that each pair intersects in exactly one point, where they cross.
A \emph{pseudo-configuration} is a pair $(\cL,V)$ such that $\cL$ is a pseudoline arrangement and $V$ is a (finite) set of points on $\bigcup\cL$ with the property that any two points in $V$ lie on a common pseudoline in $\cL$ (which is therefore unique).
A pseudo-configuration $(\cL,V)$ is in \emph{general position} if no three points in $V$ lie on a common pseudoline in $\cL$.

Let $(\cL,V)$ be a pseudo-configuration and $K$ be a Jordan curve passing through all points in $V$.
The \emph{exterior} of $K$ is the unbounded component of $\R^2\setminus K$.
We say that two points $u,v\in V$ are \emph{mutually visible} in $K$ if the pseudoline segment in $\cL$ connecting $u$ and $v$ is disjoint from the exterior of $K$.
The \emph{curve pseudo-visibility graph} $G_\cL(K,V)$ has vertex set $V$ and has an edge $uv$ for each pair of vertices $u,v\in V$ that are mutually visible in $K$.
The curve $K$ is a \emph{pseudo-polygon} on $\cL$ with vertex set $V$ if every segment of $K$ between two consecutive points in $V$ is contained in a single pseudoline in $\cL$.
Graphs of the form $G_\cL(K,V)$ where $K$ is a pseudo-polygon on $\cL$ with vertex set $V$ and $(\cL,V)$ is in general position were considered by O'Rourke and Streinu~\cite{ORS97} as \emph{pseudo-polygon visibility graphs}.
As we will see, the general position assumption is not actually restrictive in this setting.

The following two propositions imply that curve pseudo-visibility graphs are exactly the induced subgraphs of pseudo-polygon visibility graphs.
First we find a pseudo-polygon, and then we take care of the general position assumption.

\begin{proposition}
\label{prop:pseudo-polygon}
For every curve pseudo-visibility graph\/ $G=G_\cL(K,V)$, there exist a pseudo-configuration\/ $(\cL',V')$ and a pseudo-polygon\/ $K'$ on\/ $\cL'$ with vertex set\/ $V'$ such that\/ $\cL\subseteq\cL'$, $V\subseteq V'$, the points in\/ $V$ occur in the same cyclic order on\/ $K'$ as on\/ $K$, and\/ $G$ is the subgraph of\/ $G_{\cL'}(K',V')$ induced on\/ $V$.
\end{proposition}

\begin{proof}
We can assume that $K$ intersects $\bigcup\cL$ only finitely many times.
To see this, consider the finite plane graph $H$ with a vertex for each intersection point of two pseudolines in $\cL$ (including the points in $V$) and with an edge for each pseudoline segment in $\cL$ connecting two vertices and passing through no other vertex.
Let $H'$ be the vertex-spanning subgraph of $H$ obtained by including only the edges whose pseudoline segment is disjoint from the exterior of $K$.
Thus $K$ is contained in the closure of the outer (unbounded) face of $H'$.
By following the boundary of this outer face very closely (and making thin connections between connected components of the boundary if $H'$ is disconnected), we can choose $K$ to intersect $\bigcup\cL$ only finitely many times while preserving the graph $G_\cL(K,V)$ and the order of points on $K$.

Let $(\cL^*,V^*)$ be a pseudo-configuration such that $\cL\subset\cL^*$, $V\subset V^*\subset K$, every open segment of $K$ connecting two points in $\bigcup\cL$ contains a point in $V^*\setminus\bigcup\cL$, each point in $V^*$ lies on at least two pseudolines in $\cL^*$, and $|V^*|\geq 3$; we first select $V^*$ and then extend $\cL$ to $\cL^*$ using Levi's extension lemma~\cite{Levi1926}.
As before, we can assume that $K$ intersects $\bigcup\cL^*$ only finitely many times.
We further assume that $K$ has the minimum number of intersection points with $\bigcup\cL^*$ among all Jordan curves $K^*$ that pass through all of the points in $V^*$ in the same order as $K$ and are such that $G=G_\cL(K^*,V)$.

Let $V'=K\cap\bigcup\cL^*$.
In particular, $V^*\subseteq V'$.
We extend $\cL^*$ to a family of pseudolines $\cL'$ such that $(\cL',V')$ is a pseudo-configuration, using Levi's extension lemma~\cite{Levi1926}.
For any two points $u,v\in V'$ consecutive on $K$, let $K_{uv}$ be the segment of $K$ between $u$ and $v$ (which is internally disjoint from $\bigcup\cL^*$), let $L'_{uv}$ be the pseudoline in $\cL'$ passing through $u$ and $v$, let $K'_{uv}$ be the segment $uv$ of $L'_{uv}$, and let $E_{uv}$ be the unbounded component of $\R^2\setminus(K_{uv}\cup K'_{uv})$.
To construct $K'$, we replace $K_{uv}$ by $K'_{uv}$ for every pair of points $u,v\in V'$ consecutive on $K$.
Since any pseudoline in $\cL^*$ intersecting $K'_{uv}$ needs to intersect $K_{uv}$, every pseudoline in $\cL^*\setminus\{L'_{uv}\}$ is fully contained in $E_{uv}\cup\{u,v\}$.
Consequently, since each point in $V^*$ lies on at least two pseudolines in $\cL^*$, we have $V^*\subset E_{uv}\cup\{u,v\}$.

We claim that $V'\subset E_{uv}\cup\{u,v\}$ as well.
If $L'_{uv}\notin\cL^*$, then indeed $V'\subset\bigcup\cL^*\subset E_{uv}\cup\{u,v\}$.
Now, suppose $L'_{uv}\in\cL^*$.
We have $u\notin\bigcup\cL$ or $v\notin\bigcup\cL$ by the choice of $V^*$, and thus $L'_{uv}\notin\cL$.
Suppose $K\setminus K_{uv}\not\subset E_{uv}$.
Since $\bigcup\cL\subseteq\bigcup(\cL^*\setminus\{L'_{uv}\})\subset E_{uv}\cup\{u,v\}$, $V^*\subset E_{uv}\cup\{u,v\}$, and $K\setminus K_{uv}$ is disjoint from $K_{uv}$, the parts of $K\setminus K_{uv}$ not lying in $E_{uv}$ can be moved into $E_{uv}$ decreasing the number of intersection points with $\bigcup\cL^*$ (as $|V^*|\geq 3$) while preserving the graph $G_\cL(K,V)$, which contradicts the choice of $K$.
Thus $V'\subset(K\setminus K_{uv})\cup\{u,v\}\subset E_{uv}\cup\{u,v\}$ when $L'_{uv}\in\cL^*$.

For any two pairs $u,v\in V'$ and $u',v'\in V'$ of consecutive points on $K$, if the internal parts of $K'_{uv}$ and $K'_{u'v'}$ intersect, then the four points $u,u',v,v'$ occur in this or the reverse order on the boundary of $(E_{uv}\cup\{u,v\})\cap(E_{u'v'}\cup\{u',v'\})$, so the internal parts of $K_{uv}$ and $K_{u'v'}$ intersect, which is impossible.
Thus $K'$ is a Jordan curve---a pseudo-polygon on $\cL'$ with vertex set $V'$.
Furthermore, $\bigcup\cL\subset\bigcap_{uv}(E_{uv}\cup\{u,v\})$, which implies $G_\cL(K',V)=G_\cL(K,V)$.
\end{proof}

\begin{proposition}
\label{prop:general-position}
For every curve pseudo-visibility graph\/ $G=G_\cL(K,V)$, there exist a pseudo-configuration\/ $(\cL',V')$ in general position and a pseudo-polygon\/ $K'$ on\/ $\cL'$ with vertex set\/ $V'$ such that\/ $V\subseteq V'$, the points in\/ $V$ occur in the same cyclic order on\/ $K'$ as on\/ $K$, and\/ $G$ is the subgraph of\/ $G_{\cL'}(K',V')$ induced on\/ $V$.
\end{proposition}

\begin{proof}
By Proposition~\ref{prop:pseudo-polygon}, we can assume without loss of generality that $K$ is a pseudo-polygon on $\cL$.
Suppose there is a pseudoline $L$ in $\cL$ passing through more than two points in $V$.
We show that $L$ can be replaced in $\cL$ by a bunch $\cB_L$ of pseudolines in a small neighbourhood of $L$ so that the set $(\cL\setminus\{L\})\cup\cB_L$ is a pseudoline arrangement and the following conditions hold for any two distinct points $u,v\in V\cap L$.
\begin{enumerate}
\item There is a pseudoline $L_{uv}\in\cB_L$ passing through $u$, $v$, and no other points in $V$.
\item If $u$ and $v$ are consecutive points of $V\cap L$ on $L$, then the segment $uv$ of $L_{uv}$ coincides with the segment $uv$ of $L$.
\item If the segment $uv$ of $L$ is disjoint from the exterior of $K$, then so is the segment $uv$ of $L_{uv}$.
\item If the segment $uv$ of $L$ intersects the exterior of $K$, then so does the segment $uv$ of $L_{uv}$.
\end{enumerate}
Condition (4) is automatically satisfied whenever we make $\cB_L$ lie in a sufficiently small neighbourhood of $L$.
Applying this replacement repeatedly for every such pseudoline $L$ yields a claimed pseudoline arrangement $\cL'$.

For the replacement step, assume without loss of generality that $L$ is a vertical line (by applying an appropriate homeomorphism of the plane before and the inverse homeomorphism after the step).
Enumerate the points in $V\cap L$ as $v_0,\ldots,v_k$ from bottom to top.
Let $C$ be the circle with vertical diameter $v_0v_k$.
Let $v'_0=v_0$ and $v'_k=v_k$.
For $0<i<k$, let $H_i$ be the horizontal line through $v_i$, and let $v'_i$ be the left/the right/any intersection point of $C$ and $H_i$ if the exterior of $K$ touches $v_i$ from the left side/the right side/both sides of the vertical line $L$ (respectively).
For $0\leq i<j\leq k$, let $L'_{i,j}$ be the straight line passing through $v'_i$ and $v'_j$.
The bundle $\cB_L$ is obtained by ``flattening'' the family of lines $\{L_{i,j}\}_{0\leq i<j\leq k}$ horizontally to fit it in a small neighbourhood of $L$ and performing local horizontal shifts to guarantee conditions (1) and (2); condition (3) then follows.
See Figure~\ref{fig:bundle} for an illustration.
\end{proof}

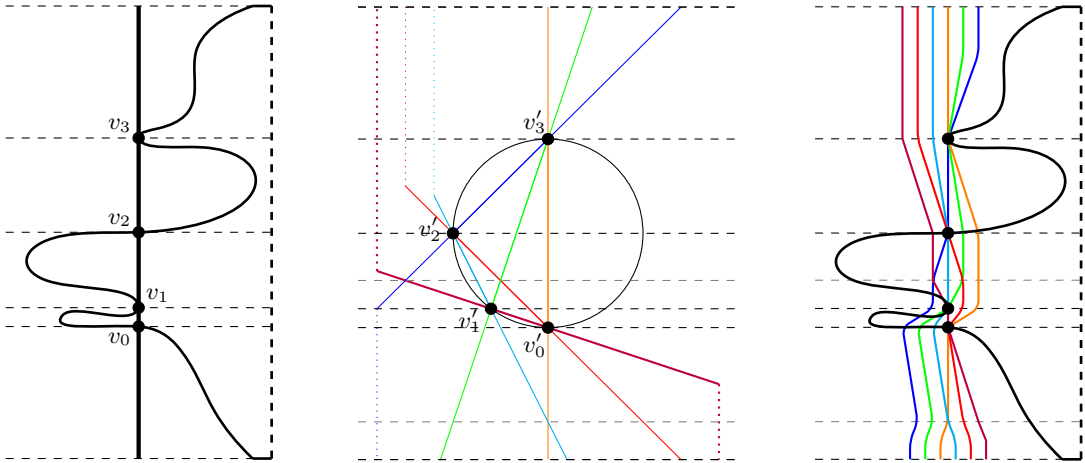
\begin{figure}[t]
\centering
\begin{tikzpicture}[xscale=0.5, yscale=0.25, every node/.style={inner sep=1.5, outer sep=0}, vtx/.style={draw, circle, fill=black}]
  \draw[thin, dashed] (-3.5,-12) -- (3.5,-12);
  \draw[thin, dashed] (-3.5,-5) -- (3.5,-5);
  \draw[thin, dashed] (-3.5,-4) -- (3.5,-4);
  \draw[thin, dashed] (-3.5,0) -- (3.5,0);
  \draw[thin, dashed] (-3.5,5) -- (3.5,5);
  \draw[thin, dashed] (-3.5,12) -- (3.5,12);
  \coordinate (v1) at (0,-5) {};
  \coordinate (v2) at (0,-4) {};
  \coordinate (v3) at (0,0) {};
  \coordinate (v4) at (0,5) {};
  \coordinate (v0v1_1) at (3,-12) {};
  \coordinate (v0v1_2) at (2,-9.5) {};
  \coordinate (v0v1_3) at (1,-6) {};
  \coordinate (v1) at (0,-5) {};
  \coordinate (v1v2_1) at (-2,-4.4) {};
  \coordinate (v2) at (0,-4) {};
  \coordinate (v2v3_1) at (0,-3.9) {};
  \coordinate (v2v3_2) at (-2.5,-2.5) {};
  \coordinate (v2v3_3) at (-2.5,-0.5) {};
  \coordinate (v3) at (0,0) {};
  \coordinate (v3v4_1) at (1.7,0.5) {};
  \coordinate (v3v4_2) at (1.7,4.5) {};
  \coordinate (v4) at (0,5) {};
  \coordinate (v4v5_1) at (0,5.1) {};
  \coordinate (v4v5_2) at (0.5,5.5) {};
  \coordinate (v4v5_3) at (1.5,9.5) {};
  \coordinate (v4v5_4) at (3,12) {};
  \draw[ultra thick] (0,-12) -- (0,12);
  \draw[line width=1pt, use Hobby shortcut] (v0v1_1)..(v0v1_2)..(v0v1_3)..(v1)..(v1v2_1)..(v2)..(v2v3_1)..(v2v3_2)..(v2v3_3)..(v3)..(v3v4_1)..(v3v4_2)..(v4)..(v4v5_1)..(v4v5_2)..(v4v5_3)..(v4v5_4); 
  \draw[line width=1pt] (3,-12) -- (3.5,-12);
  \draw[line width=1pt] (3,12) -- (3.5,12);
  \draw[line width=1pt, dashed] (3.5,-12) -- (3.5,12);
  \node[vtx] at (v1) {};
  \node[vtx] at (v2) {};
  \node[vtx] at (v3) {};
  \node[vtx] at (v4) {};
  \begin{footnotesize}
  \node at ($(v1)+(-0.5,-0.1)$) [below] {$v_0$};
  \node at ($(v2)+(0.5,0.1)$) [above] {$v_1$};
  \node at ($(v3)+(-0.5,0.1)$) [above] {$v_2$};
  \node at ($(v4)+(-0.5,0.1)$) [above] {$v_3$};
  \end{footnotesize}
\end{tikzpicture}\hskip 1cm
\begin{tikzpicture}[scale=0.25, every node/.style={inner sep=1.5, outer sep=0}, vtx/.style={draw, circle, fill=black}]
  \draw (0,0) circle (5cm);
  \coordinate (v1) at (0,-5) {};
  \coordinate (v2) at (-3,-4) {};
  \coordinate (v3) at (-5,0) {};
  \coordinate (v4) at (0,5) {};
  \draw[dashed] (-10,-12) -- (10,-12);
  \draw[dashed] (-10,12) -- (10,12);
  \draw[purple] (v1) -- (v2);
  \draw[thick, purple] (-9,-2) -- (9,-8);
  \draw[thick, purple, dotted] (-9,-2) -- (-9,12);
  \draw[thick, purple, dotted] (9,-8) -- (9,-12);
  \draw[red] (v1) -- (v3);
  \draw[red] (7,-12) -- (-7.5,2.5);
  \draw[red, dotted] (-7.5,2.5) -- (-7.5,12);
  \draw[orange] (v1) -- (v4);
  \draw[orange] (0,-12) -- (0,12);
  \draw[cyan] (v2) -- (v3);
  \draw[cyan] (1,-12) -- (-6,2);
  \draw[cyan, dotted] (-6,2) -- (-6,12);
  \draw[green] (v2) -- (v4);
  \draw[green] (-5.66,-12) -- (2.33,12);
  \draw[blue] (v3) -- (v4);
  \draw[blue] (-9,-4) -- (7,12);
  \draw[blue, dotted] (-9,-4) -- (-9,-12);
  \draw[thin, dashed] (-10,-12) -- (10,-12);
  \draw[thin, dashed, gray] (-10,-10) -- (10,-10);
  \draw[thin, dashed] (-10,-5) -- (10,-5);
  \draw[thin, dashed] (-10,-4) -- (10,-4);
  \draw[thin, dashed, gray] (-10,-2.5) -- (10,-2.5);
  \draw[thin, dashed] (-10,0) -- (10,0);
  \draw[thin, dashed] (-10,5) -- (10,5);
  \draw[thin, dashed] (-10,12) -- (10,12);
  \node[vtx] at (v1) {};
  \node[vtx] at (v2) {};
  \node[vtx] at (v3) {};
  \node[vtx] at (v4) {};
  \begin{footnotesize}
  \node at ($(v1)+(-0.7,-0.1)$) [below] {$v'_0$};
  \node at ($(v2)+(-1.1,0.2)$) [below] {$v'_1$};
  \node at ($(v3)+(-1.2,-0.6)$) [above] {$v'_2$};
  \node at ($(v4)+(-0.7,0.1)$) [above] {$v'_3$};
  \end{footnotesize}
\end{tikzpicture}\hskip 1cm
\begin{tikzpicture}[xscale=0.5, yscale=0.25, every node/.style={inner sep=1.5, outer sep=0}, vtx/.style={draw, circle, fill=black}]
  \draw[thin, dashed] (-3.5,-12) -- (3.5,-12);
  \draw[thin, dashed, gray] (-3.5,-10) -- (3.5,-10);
  \draw[thin, dashed] (-3.5,-5) -- (3.5,-5);
  \draw[thin, dashed] (-3.5,-4) -- (3.5,-4);
  \draw[thin, dashed, gray] (-3.5,-2.5) -- (3.5,-2.5);
  \draw[thin, dashed] (-3.5,0) -- (3.5,0);
  \draw[thin, dashed] (-3.5,5) -- (3.5,5);
  \draw[thin, dashed] (-3.5,12) -- (3.5,12);
  \coordinate (l0M) at (0,-12) {};
  \coordinate (l0x-3) at (-1,-12) {};
  \coordinate (l0x-2) at (-0.6,-12) {};
  \coordinate (l0x-1) at (-0.2,-12) {};
  \coordinate (l0x1) at (0.2,-12) {};
  \coordinate (l0x2) at (0.6,-12) {};
  \coordinate (l0x3) at (1,-12) {};
  \coordinate (l'0M) at (0,-11) {};
  \coordinate (l'0x-3) at (-1,-11) {};
  \coordinate (l'0x-2) at (-0.6,-11) {};
  \coordinate (l'0x-1) at (-0.2,-11) {};
  \coordinate (l'0x1) at (0.2,-11) {};
  \coordinate (l'0x2) at (0.6,-11) {};
  \coordinate (l'0x3) at (1,-11) {};
  \coordinate (l1M) at (0,-10) {};
  \coordinate (l1x-2) at (-0.8,-10) {};
  \coordinate (l1x-1) at (-0.4,-10) {};
  \coordinate (l1x0) at (0,-10) {};
  \coordinate (l1x1) at (0.4,-10) {};
  \coordinate (l1x2) at (0.8,-10) {};
  \coordinate (l2M) at (0,-5) {};
  \coordinate (l2x-3) at (-1.2,-5) {};
  \coordinate (l2x-2) at (-0.8,-5) {};
  \coordinate (l2x-1) at (-0.4,-5) {};
  \coordinate (l2x0) at (0,-5) {};
  \coordinate (l2x1) at (0.4,-5) {};
  \coordinate (l3M) at (0,-4) {};
  \coordinate (l3x-3) at (-1.2,-4) {};
  \coordinate (l3x-2) at (-0.8,-4) {};
  \coordinate (l3x-1) at (-0.4,-4) {};
  \coordinate (l3x0) at (0,-4) {};
  \coordinate (l3x1) at (0.4,-4) {};
  \coordinate (l3x2) at (0.8,-4) {};
  \coordinate (l4M) at (0,-2.5) {};
  \coordinate (l4x0) at (0,-2.5) {};
  \coordinate (l4x-1) at (-0.4,-2.5) {};
  \coordinate (l4x1) at (0.4,-2.5) {};
  \coordinate (l4x2) at (0.8,-2.5) {};
  \coordinate (l5M) at (0,0) {};
  \coordinate (l5x0) at (0,0) {};
  \coordinate (l5x-1) at (-0.4,0) {};
  \coordinate (l5x1) at (0.4,0) {};
  \coordinate (l5x2) at (0.8,0) {};
  \coordinate (l6M) at (0,5) {};
  \coordinate (l6x-3) at (-1.2,5) {};
  \coordinate (l6x-2) at (-0.8,5) {};
  \coordinate (l6x-1) at (-0.4,5) {};
  \coordinate (l6x0) at (0,5) {};
  \coordinate (l6x1) at (0.4,5) {};
  \coordinate (l6x2) at (0.8,5) {};
  \coordinate (l7M) at (0,12) {};
  \coordinate (l7x-3) at (-1.2,12) {};
  \coordinate (l7x-2) at (-0.8,12) {};
  \coordinate (l7x-1) at (-0.4,12) {};
  \coordinate (l7x1) at (0,12) {};
  \coordinate (l7x2) at (0.4,12) {};
  \coordinate (l7x3) at (0.8,12) {};
  \coordinate (l'7M) at (0,9.5) {};
  \coordinate (l'7x-3) at (-1.2,9.5) {};
  \coordinate (l'7x-2) at (-0.8,9.5) {};
  \coordinate (l'7x-1) at (-0.4,9.5) {};
  \coordinate (l'7x1) at (0,9.5) {};
  \coordinate (l'7x2) at (0.4,9.5) {};
  \coordinate (l'7x3) at (0.8,9.5) {};
  \draw[thick, blue, rounded corners=1mm] (l0x-3)--(l'0x-3)--(l1x-2)--(l2x-3)--(l3x-1)--(l4x-1)--(l5x0)--(l6x0)--(l'7x3)--(l7x3);
  \draw[thick, green, rounded corners=1mm] (l0x-2)--(l'0x-2)--(l1x-1)--(l2x-2)--(l3x0)--(l4x1)--(l5x1)--(l6x0)--(l'7x2)--(l7x2);
  \draw[thick, orange, rounded corners=1mm] (l0x-1)--(l'0x-1)--(l1x0)--(l2x0)--(l3x2)--(l4x2)--(l5x2)--(l6x0)--(l'7x1)--(l7x1);
  \draw[thick, cyan, rounded corners=1mm] (l0x1)--(l'0x1)--(l1x0)--(l2x-1)--(l3x0)--(l4x0)--(l5x0)--(l6x-1)--(l'7x-1)--(l7x-1);
  \draw[thick, red, rounded corners=1mm] (l0x2)--(l'0x2)--(l1x1)--(l2x0)--(l3x1)--(l4x1)--(l5x0)--(l6x-2)--(l'7x-2)--(l7x-2);
  \draw[thick, purple, use Hobby shortcut] (l0x3)--(l'0x3)--(l1x2)--(l2x0)--(l3x0)--(l4x-1)--(l5x-1)--(l6x-3)--(l'7x-3)--(l7x-3);
  \coordinate (v1) at (0,-5) {};
  \coordinate (v2) at (0,-4) {};
  \coordinate (v3) at (0,0) {};
  \coordinate (v4) at (0,5) {};
  \coordinate (v0v1_1) at (3,-12) {};
  \coordinate (v0v1_2) at (2,-9.5) {};
  \coordinate (v0v1_3) at (1,-6) {};
  \coordinate (v1) at (0,-5) {};
  \coordinate (v1v2_1) at (-2,-4.4) {};
  \coordinate (v2) at (0,-4) {};
  \coordinate (v2v3_1) at (0,-3.9) {};
  \coordinate (v2v3_2) at (-2.5,-2.5) {};
  \coordinate (v2v3_3) at (-2.5,-0.5) {};
  \coordinate (v3) at (0,0) {};
  \coordinate (v3v4_1) at (1.7,0.5) {};
  \coordinate (v3v4_2) at (1.7,4.5) {};
  \coordinate (v4) at (0,5) {};
  \coordinate (v4v5_1) at (0,5.1) {};
  \coordinate (v4v5_2) at (0.5,5.5) {};
  \coordinate (v4v5_3) at (1.5,9.5) {};
  \coordinate (v4v5_4) at (3,12) {};
  \draw[line width=1pt, use Hobby shortcut] (v0v1_1)..(v0v1_2)..(v0v1_3)..(v1)..(v1v2_1)..(v2)..(v2v3_1)..(v2v3_2)..(v2v3_3)..(v3)..(v3v4_1)..(v3v4_2)..(v4)..(v4v5_1)..(v4v5_2)..(v4v5_3)..(v4v5_4); 
  \node[vtx] at (v1) {};
  \node[vtx] at (v2) {};
  \node[vtx] at (v3) {};
  \node[vtx] at (v4) {};
  \draw[line width=1pt] (3,-12) -- (3.5,-12);
  \draw[line width=1pt] (3,12) -- (3.5,12);
  \draw[line width=1pt, dashed] (3.5,-12) -- (3.5,12);
\end{tikzpicture}
\caption{Replacing $L$ with a bundle of pseudo-lines $\cB_L$ in the proof of Proposition \ref{prop:general-position}.}
\label{fig:bundle}
\end{figure}

Recall that an \emph{ordered graph} is a tuple $(G,{\prec})$ such that $G$ is a graph and $\prec$ is a linear order on its vertex set.
While it is more convenient to work with linear orders, the points on a Jordan curve are really ordered cyclically.
A \emph{rotation} of a linear order $\prec$ is any linear order obtained from $\prec$ by repeatedly making the largest element the smallest.
We think of any finite set of points $V$ on a Jordan curve $K$ as being ordered counterclockwise around $K$, as in Figure~\ref{fig:graph_classes} (bottom-left).
We call any linear order which begins at an arbitrary point in $V$ and then follows $K$ in the counterclockwise direction a \emph{natural order} of $V$ on $K$.
A curve pseudo-visibility graph $G_\cL(K,V)$ along with a natural order of $V$ on $K$ forms an \emph{ordered curve pseudo-visibility graph}.

If $(G,{\prec})$ is an ordered graph with vertices $a\prec b\prec c\prec d$ and edges $ac$ and $bd$, we say that $ac$ \emph{crosses} $bd$, that $ac$ and $bd$ are \emph{crossing edges}, and that $bd$ is \emph{crossed} by $ac$.
Two edges which are not crossing are called \emph{non-crossing}.
The property that a pair of edges is crossing/non-crossing is preserved under rotation (since, using this terminology, we do not specify which edge crosses the other).
In particular, it is well defined for an ordered curve pseudo-visibility graph regardless of the choice of a natural ordering.

\begin{lemma}
\label{lem:rep}
For a curve pseudo-visibility graph\/ $G_\cL(K,V)$ with\/ $(\cL,V)$ in general position, two distinct edges\/ $uv$ and\/ $xy$ are crossing if and only if the open segments\/ $uv$ and\/ $xy$ of pseudolines in\/ $\cL$ intersect.
\end{lemma}

\begin{proof}
If $uv$ and $xy$ are crossing edges, then the open segments $uv$ and $xy$ must intersect; otherwise $K$ along with $uv$ and $xy$ give an outerplanar drawing of $K_4$, which is impossible.
If $uv$ and $xy$ are non-crossing while the open segments $uv$ and $xy$ intersect, then we can again obtain an outerplanar drawing of $K_4$ by re-connecting $uv$ and $xy$ in a sufficiently small neighbourhood of their unique intersection point---a contradiction.
\end{proof}

\section{Obstructions for curve pseudo-visibility graphs}
\label{sec:obstructions}

In this section, we discuss the obstructions mentioned in the introduction: the class $\cH$ and the class of ordered holes.
Ghosh~\cite{Ghosh97} observed that these are obstructions for polygon visibility graphs, and related obstructions in the pseudo-visibility setting appear in the works of Abello and Kumar~\cite{AKMatroid02} and O'Rourke and Streinu~\cite{ORS97}.

If two vertices $u$ and $v$ are non-adjacent in a curve pseudo-visibility graph $G_\cL(K,V)$, then it is because at least one of two parts of $K$ between $u$ and $v$ ``blocks'' visibility between them.
The intuition behind the next definition is that if there is a ``crossing sequence'' from $u$ to $v$, then the part of $K$ from $u$ to $v$ (counterclockwise) cannot ``block'' visibility between $u$ and $v$.

\begin{figure}[t]
\centering
\begin{tikzpicture}[scale=.8, every node/.style={inner sep=2, outer sep=0, draw, circle, fill=gray}, >=latex]
  \draw[ultra thin] (0,0) circle (2) {};
  \node[label=left:$u$] (u) at (180:2) {};
  \node[label=right:$v$] (v) at (0:2) {};
  \foreach\i in {1,...,4}{
    \node[inner sep=1.5] (A\i) at (170+\i*36:2) {};
    \node[inner sep=1.5] (B\i) at (190+\i*36:2) {};
  }
  \draw[very thick] (u) -- (B1);
  \draw[very thick] (A1) -- (B2);
  \draw[very thick] (A2) -- (B3);
  \draw[very thick] (A3) -- (B4);
  \draw[very thick] (A4) -- (v);
  \draw[->] (190:2.4) arc (190:230:2.4);
\end{tikzpicture}\hskip 2cm
\begin{tikzpicture}[scale=.8, every node/.style={inner sep=2, outer sep=0, draw, circle, fill=gray}, >=latex]
  \draw[ultra thin] (0,0) circle (2) {};
  \node[label=left:$u$] (u) at (180:2) {};
  \node[label=right:$v$] (v) at (0:2) {};
  \node[inner sep=1.5] (A1) at (45:2) {};
  \node[inner sep=1.5] (A2) at (135:2) {};
  \node[inner sep=1.5] (B) at (90:2) {};
  \draw[very thick] (v) -- (B);
  \draw[very thick] (A1) -- (A2);
  \draw[very thick] (B) -- (u);
  \draw[->] (10:2.4) arc (10:50:2.4);
\end{tikzpicture}
\caption{A crossing sequence from $u$ to $v$ (left) and from $v$ to $u$ (right).}
\label{fig:crossing_sequence}
\end{figure}
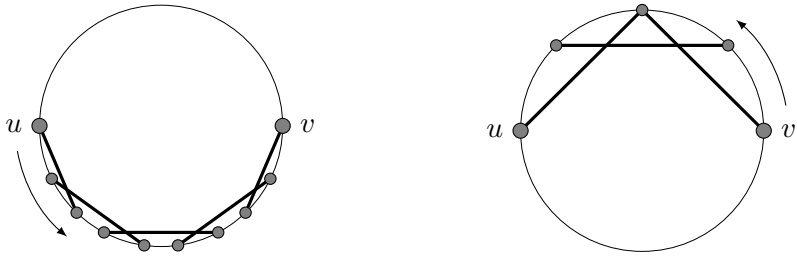

Let $u$ and $v$ be two distinct vertices in an ordered graph $(G,{\prec})$.
If $u\prec v$, then a \emph{crossing sequence} from $u$ to $v$ is a sequence of distinct edges $e_1,\ldots,e_k$ such that $u$ is the smaller end of $e_1$, $v$ is the larger end of $e_k$, and $e_i$ crosses $e_{i+1}$ for $1\leq i<k$.
Observe that the notion of a crossing sequence is invariant under rotation of $\prec$ as long as $u\prec v$.
If $v\prec u$, then a \emph{crossing sequence} from $u$ to $v$ is a crossing sequence from $u$ to $v$ in any rotation $\prec'$ of $\prec$ such that $u\prec'v$.
These definitions should be thought of cyclically; whichever vertex is smaller, a crossing sequence from $u$ to $v$ begins at $u$ and goes counterclockwise until it hits $v$ (see Figure~\ref{fig:crossing_sequence}).
If $u$ and $v$ are adjacent, then the edge $uv$ is a crossing sequence from $u$ to $v$ and from $v$ to $u$.

\begin{lemma}
\label{lem:almostTran}
If\/ $(G,{\prec})$ is an ordered graph with vertices\/ $a\prec b\prec c\prec d$ and there are crossing sequences from\/ $a$ to\/ $c$ and from\/ $b$ to\/ $d$, then there is a crossing sequence from\/ $a$ to\/ $d$.
\end{lemma}

\begin{proof}
Let $e_1,\ldots,e_k$ and $f_1,\ldots,f_t$ be crossing sequences from $a$ to $c$ and from $b$ to $d$, respectively.
Let $e_i$ be the edge with the smallest index such that its larger end, say $v$, is greater than $b$ in $\prec$.
Let $f_j$ be the edge with the largest index such that its smaller end is less than $v$ in $\prec$.
Then $e_i$ crosses $f_j$ and $e_1,\ldots,e_i,f_j,\ldots,f_t$ is a crossing sequence from $a$ to $d$.
\end{proof}

The first family of obstructions, which we denote by $\cH$, is defined as follows: $\cH$ is the family of all ordered graphs containing two non-adjacent vertices $u$ and $v$ such that there exist a crossing sequence from $u$ to $v$ and a crossing sequence from $v$ to $u$.
See Figure~\ref{fig:obstructions} (left) for an illustration.
The second family of obstructions is the family of ordered holes.
An \emph{ordered hole} is an ordered graph $(H,{\prec})$ on vertex set $V(H)=\{c_1,\dots,c_k\}$, where $k\geq 4$ and $c_1\prec\cdots\prec c_k$, with edge set $E(H)=\{c_1c_2,\ldots,c_{k-1}c_k,c_kc_1\}$; see Figure~\ref{fig:obstructions} (middle).

\begin{proposition}
\label{prop:H-free}
Every ordered curve pseudo-visibility graph is\/ $\cH$-free.
\end{proposition}

\begin{proposition}
\label{prop:ordered-hole-free}
Every ordered curve pseudo-visibility graph is ordered-hole-free.
\end{proposition}

We prove Propositions \ref{prop:H-free} and~\ref{prop:ordered-hole-free} later in this section.
Before that, we show that we can test in polynomial time whether a given ordered graph is free of the considered obstructions.

\begin{proposition}
\label{prop:H-free-alg}
There is a polynomial-time algorithm which takes in an ordered graph\/ $(G,{\prec})$ and determines whether\/ $(G,{\prec})$ is\/ $\cH$-free.
\end{proposition}

\begin{proof}
It suffices to test, for any two non-adjacent vertices $u$ and $v$, whether $(G,{\prec})$ has a crossing sequence from $u$ to $v$.
We assume that $u\prec v$ after possibly performing a rotation.
We create a directed graph $\vec{H}$ with a vertex for each edge of $G$ and with an arc from $e$ to $f$ for each pair of edges of $G$ such that $e$ crosses $f$.
Then there is a crossing sequence from $u$ to $v$ in $(G,{\prec})$ if and only if there is an edge $e$ with smaller end $u$ and an edge $f$ with larger end $v$ such that $\vec{H}$ has a directed path from $e$ to $f$.
\end{proof}

\begin{proposition}
There is a polynomial-time algorithm which takes in an ordered graph\/ $(G,{\prec})$ and determines whether\/ $(G,{\prec})$ has an ordered hole.
\end{proposition}

\begin{proof}
It suffices to test, for any two adjacent vertices $u\prec v$ of $G$, whether $u$ and $v$ are the first and last vertices of an ordered hole.
This can be done by removing all vertices in a triangle with $u$ and $v$ and then testing for a directed path from $u$ to $v$ in the natural digraph.
\end{proof}

\begin{figure}[t]
\centering
\begin{tikzpicture}[scale=0.19, every node/.style={inner sep=2, outer sep=0}, vtx/.style={draw, circle, fill=gray}]
  \coordinate (A1) at (-2,8.5) {};
  \coordinate (A2) at (0,0) {};
  \coordinate (A3) at (-9,-2) {};
  \coordinate (A4) at (2,-9) {};
  \coordinate (A5) at (9,3) {};
  \coordinate (L124_B1_1) at (-2.5,14) {};
  \coordinate (L124_A1_1) at (-2.2,6) {};
  \coordinate (L124_A1_2) at (-2,4) {};
  \coordinate (L124_A2_1) at (3,-2) {};
  \coordinate (L124_A2_2) at (3,-6) {};
  \coordinate (L124_A4_1) at (1,-14) {};
  \coordinate (L13_B1_1) at (8,14) {};
  \coordinate (L13_B1_2) at (4,11) {};
  \coordinate (L13_A3_1) at (-9,-8) {};
  \coordinate (L13_A3_2) at (-10,-14) {};
  \coordinate (L325_B3_1) at (-14,-2) {};
  \coordinate (L325_A3_1) at (-7,-2.3) {};
  \coordinate (L325_A3_2) at (-4,-2) {};
  \coordinate (L325_A2_1) at (2,2) {};
  \coordinate (L325_A2_2) at (6,3.2) {};
  \coordinate (L325_A5_1) at (14,3) {};
  \coordinate (L15_B1_1) at (-14,6) {};
  \coordinate (L15_B1_2) at (-8,6.75) {};
  \coordinate (L15_A1_1) at (2,8) {};
  \coordinate (L15_A1_2) at (6,6) {};
  \coordinate (L15_A5_1) at (11,-2) {};
  \coordinate (L15_A5_2) at (14,-6) {};
  \coordinate (L34_B3_1) at (-10,14) {};
  \coordinate (L34_B3_2) at (-10,8) {};
  \coordinate (L34_A3_1) at (-7,-4) {};
  \coordinate (L34_A3_2) at (-4,-6) {};
  \coordinate (L34_A4_1) at (14,-12) {};
  \coordinate (L45_B4_1) at (-14,-11) {};
  \coordinate (L45_A4_1) at (5,-8) {};
  \coordinate (L45_A4_2) at (9,-3.5) {};
  \coordinate (L45_A4_3) at (9.25,0) {};
  \coordinate (L45_A5_1) at (14,10) {};
  \useasboundingbox (-17,-16) rectangle (17,16);
  \draw[use Hobby shortcut] (L124_B1_1) .. (A1) .. (L124_A1_1) .. (L124_A1_2) .. (A2) .. (L124_A2_1) .. (L124_A2_2) .. (A4) .. (L124_A4_1);
  \draw[use Hobby shortcut] (L325_B3_1) .. (A3) .. (L325_A3_1) .. (L325_A3_2) .. (A2) .. (L325_A2_1) .. (L325_A2_2) .. (A5) .. (L325_A5_1);
  \draw[use Hobby shortcut] (L15_B1_1) .. (L15_B1_2) .. (A1) .. (L15_A1_1) .. (L15_A1_2) .. (A5) .. (L15_A5_1) .. (L15_A5_2);
  \draw[use Hobby shortcut] (L13_B1_1) .. (L13_B1_2) .. (A1) .. (A3) .. (L13_A3_1) .. (L13_A3_2);
  \draw[use Hobby shortcut] (L34_B3_1) .. (L34_B3_2) .. (A3) .. (L34_A3_1) .. (L34_A3_2) .. (A4) .. (L34_A4_1);
  \draw[use Hobby shortcut] (L45_B4_1) .. (A4) .. (L45_A4_1) .. (L45_A4_2) .. (L45_A4_3) .. (A5) .. (L45_A5_1);
  \draw[line width=3pt, use Hobby shortcut] (L124_B1_1) .. ([blank]A1) .. (L124_A1_1) .. (L124_A1_2) .. (A2) .. ([blank]L124_A2_1) .. ([blank]L124_A2_2) .. ([blank]A4) .. ([blank]L124_A4_1);
  \draw[line width=3pt, use Hobby shortcut] (L325_B3_1) .. ([blank]A3) .. (L325_A3_1) .. (L325_A3_2) .. (A2) .. ([blank]L325_A2_1) .. ([blank]L325_A2_2) .. ([blank]A5) .. ([blank]L325_A5_1);
  \draw[line width=3pt, use Hobby shortcut] (L15_B1_1) .. ([blank]L15_B1_2) .. ([blank]A1) .. (L15_A1_1) .. (L15_A1_2) .. (A5) .. ([blank]L15_A5_1) .. ([blank]L15_A5_2);
  \draw[line width=3pt, use Hobby shortcut] (L34_B3_1) .. ([blank]L34_B3_2) .. ([blank]A3) .. (L34_A3_1) .. (L34_A3_2) .. (A4) .. ([blank]L34_A4_1);
  \draw[line width=3pt, use Hobby shortcut] (L45_B4_1) .. ([blank]A4) .. (L45_A4_1) .. (L45_A4_2) .. (L45_A4_3) .. (A5) .. ([blank]L45_A5_1);
  \node[vtx] at (A1) {};
  \node[vtx] at (A2) {};
  \node[vtx] at (A3) {};
  \node[vtx] at (A4) {};
  \node[vtx] at (A5) {};
  \node at ($(A1)+(-1,0.5)$) [above] {$1$};
  \node at ($(A2)+(0,-0.65)$) [below] {$2$};
  \node at ($(A3)+(-1,-0.5)$) [below] {$3$};
  \node at ($(A4)+(0.8,-0.6)$) [below] {$4$};
  \node at ($(A5)+(1.2,0.35)$) [above] {$5$};
  \node at (L34_B3_1) [above] {$34$};
  \node at (L34_A4_1) [right] {$34$};
  \node at (L124_B1_1) [above] {$124$};
  \node at (L124_A4_1) [below] {$124$};
  \node at (L325_B3_1) [left] {$325$};
  \node at (L325_A5_1) [right] {$325$};
  \node at (L15_B1_1) [left] {$15$};
  \node at (L15_A5_2) [right] {$15$};
  \node at (L13_B1_1) [above] {$13$};
  \node at (L13_A3_2) [below] {$13$};
  \node at (L45_B4_1) [left] {$45$};
  \node at (L45_A5_1) [right] {$45$};
\end{tikzpicture}
\caption{A pseudo-polygon with five articulation points: $1,3,4,5$ are convex, $2$ is concave.}
\label{fig:articulation}
\end{figure}
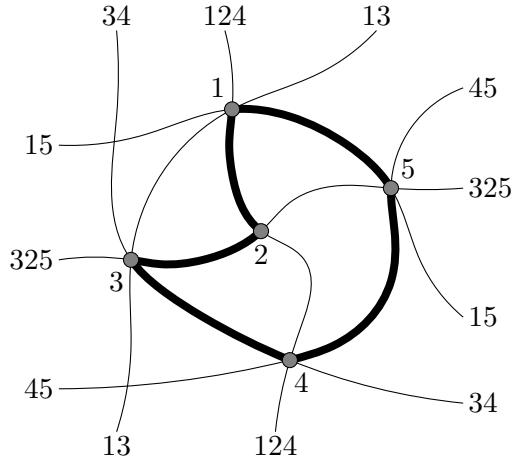

The proofs of Propositions \ref{prop:H-free} and~\ref{prop:ordered-hole-free} require some preparation.
Let $K$ be a pseudo-polygon on $\cL$.
A~\emph{segment} of $K$ is a part of $K$ that is contained in some pseudoline $L\in\cL$ and connects two distinct intersection points of $L$ with other pseudolines in $\cL$.
An \emph{articulation point} of $K$ is a point in $K$ that joins two segments of $K$ contained in distinct pseudolines in $\cL$.
Such an articulation point of $K$ is \emph{convex} if those two pseudolines extend to the exterior of $K$ at $p$, and it is \emph{concave} if they extend to the interior of $K$; see Figure~\ref{fig:articulation}.
The following lemma was proven by Arroyo, Bensmail, and Richter~\cite{arroyo20}; we provide a proof for the reader's convenience.

\begin{lemma}
\label{lem:convex}
Every pseudo-polygon on\/ $\cL$ has at least three convex articulation points.
\end{lemma}

\begin{proof}
Suppose otherwise, and choose a counterexample $K$ with as few articulation points as possible.
Since $\cL$ is a pseudoline arrangement, $K$ has at least three articulation points.
Thus, we can choose consecutive articulation points $p_1$, $p_2$, and $p_3$ which occur on $K$ in that order counterclockwise so that $p_2$ is concave and if any articulation point is convex, then $p_3$ is convex.
Now, walk from $p_1$ towards $p_2$ along the pseudoline $L\in\cL$ passing through $p_1$ and $p_2$, and continue walking on $L$ beyond $p_2$ (through the interior of $K$, as $p_2$ is concave) until hitting $K$ at a point $a\in K\cap L$.
Let $K'$ denote the pseudo-polygon formed by the segment $p_2a$ of $L$ and the part of $K$ from $a$ to $p_2$ counterclockwise.
It follows that $K'$ has at most two convex articulation points and has fewer articulation points than $K$, because at most one articulation point, $a$, is gained, and the articulation points $p_2$ and $p_3$ of $K$ are lost.
This is a contradiction, completing the proof.
\end{proof}

\begin{lemma}
\label{lem:forcing}
Let\/ $K$ be a pseudo-polygon on\/ $\cL$.
Let\/ $u$ and\/ $v$ be distinct points on\/ $K$ such that\/ $\cL$ contains a pseudoline\/ $L$ passing through\/ $u$ and\/ $v$.
If all articulation points of\/ $K$ other than possibly\/ $u$ and\/ $v$ are convex, then the segment\/ $uv$ of\/ $L$ is disjoint from the exterior of\/ $K$.
\end{lemma}

\begin{proof}
First, suppose that neither $u$ nor $v$ is a concave articulation point of $K$.
Suppose for the sake of contradiction that the segment $uv$ of $L$ is not disjoint from the exterior of $K$, and let $xy$ be a maximal subsegment of it with internal part contained in the exterior of $K$.
Thus $x,y\in K$.
The segment $xy$ of $L$ together with one of the parts of $K$ between $x$ and $y$ forms a pseudo-polygon on $\cL$ with interior contained in the exterior of $K$ and with at most two convex articulation points: $x$ and $y$.
This contradicts Lemma~\ref{lem:convex}.

Now, suppose that $u$ is a concave articulation point of $K$ while $v$ is not.
Let $L'$ be a pseudoline containing one of the two segments of $K$ incident to $u$.
Follow $L'$ from $u$ in the other direction (towards the interior of $K$) until it hits $K$ at some point $x$.
Let $K'$ be a pseudo-polygon formed by the segment $ux$ of $L'$ and the part of $K$ between $x$ and $u$ that contains the point $v$.
Thus $x$ is a convex articulation point of $K'$, $u$ is no longer a concave articulation point of $K'$, and every other articulation point of $K$ that lies on $K'$ remains convex on $K'$.
Therefore, as we shown in the first case, the segment $uv$ of $L$ is disjoint from the exterior of $K'$, so it is disjoint from the exterior of $K$.

The argument is analogous if $v$ is a concave articulation point of $K$, except that when $u$ is also a concave articulation point of $K$, then we apply the same argument as above to reduce to the case that only one of $u,v$ is a concave articulation point of $K$.
\end{proof}

\begin{proof}[Proof of Proposition \ref{prop:H-free}]
Let $G=G_\cL(K,V)$ be a curve pseudo-visibility graph and $\prec$ be a natural order of $V$ on $K$.
By Proposition~\ref{prop:general-position}, we can assume that $(\cL,V)$ is in general position.
For an edge $e=uv\in E(G)$, let $\ell_e$ denote the open segment $uv$ of the pseudoline in $\cL$ passing through $u$ and $v$.
Suppose that there are $u,v\in V$ with $u\prec v$ such that there are crossing sequences $e_1,\ldots,e_k$ from $u$ to $v$ and $f_1,\ldots,f_t$ from $v$ to $u$.
Choose the two crossing sequences so that $k+t$ is minimum.
We need to show that $uv$ is an edge of $G$.

By minimality and Lemmas \ref{lem:rep} and~\ref{lem:almostTran}, for $1\leq i<j\leq k$, the segments $\ell_{e_i}$ and $\ell_{e_j}$ intersect if and only if $j=i+1$, and likewise for the crossing sequence $f_1,\ldots,f_t$.
Also by Lemma~\ref{lem:rep}, each $\ell_{e_i}$ is disjoint from each $\ell_{f_j}$.
Therefore, by beginning at $u$ and walking along $\ell_{e_1}$ until its unique intersection with $\ell_{e_2}$ is reached, then turning left and walking along $\ell_{e_2}$ until either $v$ or its unique intersection with $\ell_{e_3}$ is reached, and so on, we can find an open curve $K_1\subseteq\bigcup_{i=1}^k\ell_{e_i}$ with ends $u$ and $v$.
Likewise we can find an open curve $K_2\subseteq\bigcup_{j=1}^t\ell_{f_j}$ with ends $v$ and $u$.
Let $K=K_1\cup K_2\cup\{u,v\}$.
It follows that $K$ is a pseudo-polygon on $\cL$ and all articulation points of $K$ except possibly $u$ and $v$ are convex.
Therefore, by Lemma~\ref{lem:forcing}, the segment $\ell_{uv}$ is disjoint from the exterior of $K$, so $uv\in E(G)$.
\end{proof}

\begin{proof}[Proof of Proposition \ref{prop:ordered-hole-free}]
When $K$ is a Jordan curve and $p,q\in K$, we write $K[p,q]$ and $K(p,q)$ for the closed and the open segment (respectively) of $K$ from $p$ to $q$ in the counterclockwise direction.
When $L$ is a pseudoline and $p,q\in L$, we write $L[p,q]$ for the closed segment of $L$ connecting $p$ and $q$.

Let $G=G_\cL(K,V)$ be a curve pseudo-visibility graph.
By Proposition~\ref{prop:general-position}, we can assume without loss of generality that $(\cL,V)$ is in general position.
Suppose for the sake of contradiction that $G$ has an induced cycle of length at least four on vertices $p_1,\ldots,p_k$ with $p_1\prec\cdots\prec p_k$.
We write indices cyclically, so that $p_{k+1}=p_1$, $p_{k+2}=p_2$, and so on.
Let $K$ be pseudo-polygon on $\cL$ formed by the pseudoline segments $p_ip_{i+1}$ with $1\leq i\leq k$.
It follows from Lemma~\ref{lem:rep} that $K$ is a pseudo-polygon with articulation points $p_1,\ldots,p_k$ in cyclic order (and no other articulation points).
By Lemma~\ref{lem:convex}, $K$ has a convex articulation point.
Up to rotation, we can assume that $p_2$ is convex.

By Lemma~\ref{lem:forcing}, since $p_1p_3$ is not an edge of $G$ and $p_2$ is convex, there is a concave articulation point of $K$ among $p_4,\ldots,p_k$.
We now choose a new cycle $K'$ with similar properties, but which is also ``minimal''.
We will then use the existence of a concave articulation point of $K'$ to reach a contradiction to minimality.
For $1\leq i\leq k$, let $L_i$ be the pseudoline in $\cL$ passing through $p_i$ and $p_{i+1}$.
Choose a pseudo-polygon $K'$ on $\{L_i\}_{1\leq i\leq k}$ so that
\begin{enumeratei}
\item $K[p_1,p_3]\subseteq K'$ and $K'$ is disjoint from the exterior of $K$,
\item every concave articulation point of $K'$ belongs to $\{p_1,\ldots,p_k\}$, and
\item subject to conditions (i) and (ii), the region of the plane bounded by $K'$ is minimal.
\end{enumeratei}
Such a cycle $K'$ exists, as $K$ is a candidate.

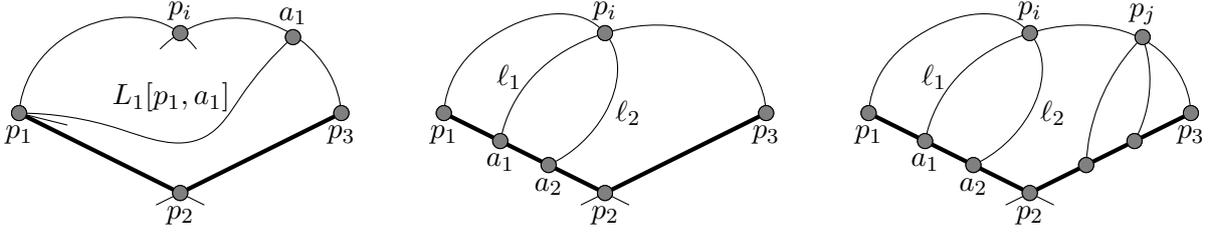
\begin{figure}[t]
\centering
\begin{tikzpicture}[scale=.53, every node/.style={inner sep=2, outer sep=0}, vtx/.style={draw, circle, fill=gray}]
  \node[vtx, label=below:$p_1$] (p1) at (-4,0) {};
  \node[vtx, label=below:$p_2$] (p2) at (0,-2) {};
  \node[vtx, label=below:$p_3$] (p3) at (4,0) {};
  \node (a) at (3.7,1) {};
  \node[vtx, label=above:$p_i$] (pi) at (0,2) {};
  \node (c0) at (-3.7,1) {};
  \node (c1) at (-2.4,2.15) {};
  \draw[ultra thick] (p1) -- (p2) -- (p3);
  \draw (-.6,-2.3) -- (p2) -- (.6,-2.3);
  \node (b) at (-.5,1.6) {};
  \node (b') at (.4,1.6) {};
  \draw[use Hobby shortcut] (p3.center) .. (a.center) .. (pi.center) .. (b.center);
  \draw[use Hobby shortcut] (p1.center) .. (c0.center) .. (c1.center) .. (pi.center) .. (b'.center);
  \node (c) at (-1.5,-.4) {};
  \node (c') at (0.6,-.6) {};
  \node[label=left:{$L_1[p_1,a_1]$}] (c'') at (1.5,.4) {};
  \node[vtx, label=above:$a_1$] (a1) at (2.8,1.9) {};
  \draw[use Hobby shortcut] (p1.center) .. (c.center) .. (c'.center) .. (c''.center) .. (a1.center);
  \draw (p1) -- (-2.8,-.3);
  \node[vtx] at (p1) {};
  \node[vtx] at (p2) {};
  \node[vtx] at (p3) {};
  \node[vtx] at (pi) {};
  \node[vtx] at (a1) {};
\end{tikzpicture}\hfill
\begin{tikzpicture}[scale=.53, every node/.style={inner sep=2, outer sep=0}, vtx/.style={draw, circle, fill=gray}]
  \node[vtx, label=below:$p_1$] (p1) at (-4,0) {};
  \node[vtx, label=below:$p_2$] (p2) at (0,-2) {};
  \node[vtx, label=below:$p_3$] (p3) at (4,0) {};
  \node (a) at (3.7,1) {};
  \node[label=left:$\ell_1$] (b2) at (-1.8,.9) {};
  \node[vtx, label=below:$a_1$] (a1) at (-2.6,-.7) {};
  \node[vtx, label=above:$p_i$] (pi) at (0,2) {};
  \node (c0) at (-3.7,1) {};
  \node (c1) at (-2.4,2.15) {};
  \node[label=right:$\ell_2$] (c2) at (0,0) {};
  \node[vtx, label=below:$a_2$] (a2) at (-1.4,-1.3) {};
  \draw[ultra thick] (p1) -- (p2) -- (p3);
  \draw (-.6,-2.3) -- (p2) -- (.6,-2.3);
  \draw[use Hobby shortcut] (p3.center) .. (a.center) .. (pi.center) .. (b2.center) .. (a1.center);
  \draw[use Hobby shortcut] (p1.center) .. (c0.center) .. (c1.center) .. (pi.center) .. (c2.center) .. (a2.center);
  \node[vtx] at (p1) {};
  \node[vtx] at (p2) {};
  \node[vtx] at (p3) {};
  \node[vtx] at (pi) {};
  \node[vtx] at (a1) {};
  \node[vtx] at (a2) {};
\end{tikzpicture}\hfill
\begin{tikzpicture}[scale=.53, every node/.style={inner sep=2, outer sep=0}, vtx/.style={draw, circle, fill=gray}]
  \node[vtx, label=below:$p_1$] (p1) at (-4,0) {};
  \node[vtx, label=below:$p_2$] (p2) at (0,-2) {};
  \node[vtx, label=below:$p_3$] (p3) at (4,0) {};
  \node (a) at (3.7,1) {};
  \node[label=left:$\ell_1$] (b2) at (-1.8,.9) {};
  \node[vtx, label=below:$a_1$] (a1) at (-2.6,-.7) {};
  \node[vtx, label=above:$p_i$] (pi) at (0,2) {};
  \node (c0) at (-3.7,1) {};
  \node (c1) at (-2.4,2.15) {};
  \node[label=right:$\ell_2$] (c2) at (0,0) {};
  \node[vtx, label=below:$a_2$] (a2) at (-1.4,-1.3) {};
  \draw[ultra thick] (p1) -- (p2) -- (p3);
  \draw (-.6,-2.3) -- (p2) -- (.6,-2.3);
  \draw[use Hobby shortcut] (p3.center) .. (a.center) .. (pi.center) .. (b2.center) .. (a1.center);
  \draw[use Hobby shortcut] (p1.center) .. (c0.center) .. (c1.center) .. (pi.center) .. (c2.center) .. (a2.center);
  \node[vtx, label=above:$p_j$] (pj) at (2.8,1.9) {};
  \node[vtx] (d1) at (2.6,-.7) {};
  \node[vtx] (d2) at (1.4,-1.3) {};
  \draw[use Hobby shortcut] (d2.center) .. (1.5,-.5) .. (pj.center);
  \draw[use Hobby shortcut] (d1.center) .. (2.9,.1) .. (pj.center);
  \node[vtx] at (p1) {};
  \node[vtx] at (p2) {};
  \node[vtx] at (p3) {};
  \node[vtx] at (pi) {};
  \node[vtx] at (a1) {};
  \node[vtx] at (a2) {};
  \node[vtx] at (pj) {};
  \node[vtx] at (d1) {};
  \node[vtx] at (d2) {};
\end{tikzpicture}
\caption{Illustrations for the proof of Proposition~\ref{prop:ordered-hole-free}, with $K'[p_1,p_3]$ in bold.}
\label{fig:orderedHole}
\end{figure}

With this choice, we still have that $K'(p_3,p_1)$ contains a concave articulation point of $K'$, as otherwise, by Lemma~\ref{lem:forcing}, the pseudoline segment $p_1p_3$ in $\cL$ would be disjoint from the exterior of $K'$ and thus from the exterior of $K$, contradicting the assumption that $p_1p_3$ is not an edge of $G$.
By (ii), there is an index $i$ with $4\leq i\leq k$ such that $p_i$ is a concave articulation point of $K'$.
Moreover, we have the following.

\begin{claim}
\label{claim:concave}
Neither $p_1$ nor $p_3$ is a concave articulation point of $K'$.
\end{claim}

\begin{proof}
Suppose that $p_1$ is a concave articulation point of $K'$.
Walk along the pseudoline $L_1$ from $p_1$ in the direction opposite to $p_2$ (towards the interior of $K'$) until hitting $K'$ at a point $a_1\in L_1\cap K'$.
Then the pseudo-polygon $K'[p_1,a_1]\cup L_1[p_1,a_1]$ contradicts the choice of $K'$; see Figure~\ref{fig:orderedHole} (top left).
\end{proof}

Recall that $p_i$ is a concave articulation point of $K'$ with $4\leq i\leq k$.
By (i) and (ii), there are open segments $\ell_1$ of $L_{i-1}$ and $\ell_2$ of $L_i$ that are both contained in the interior of $K'$, have $p_i$ as one endpoint and have the other endpoint on $K'$; see Figure~\ref{fig:orderedHole} (top right).
Let $a_1$ and $a_2$ be the other endpoints of $\ell_1$ and $\ell_2$, respectively.
We now show that, up to symmetry, the case depicted in Figure~\ref{fig:orderedHole} (top right) actually occurs.

\begin{claim}
Either $a_1,a_2\in K'(p_1,p_2)$ or $a_1,a_2\in K'(p_2,p_3)$.
\end{claim}

\begin{proof}
We have $a_1\notin K'[p_i,p_1]$, otherwise the pseudo-polygon $K'[a_1,p_i]\cup\ell_1$ would contradict the choice of $K'$.
Furthermore, $a_1\neq p_2$, because $a_1\in L_{i-1}$ and $L_{i-1}$ does not pass through $p_2$ by the general position assumption.
By a symmetric argument for $a_2$, we obtain that $a_1,a_2\in K'(p_1,p_3)\setminus\{p_2\}$.
Finally, it is not possible that $a_1\in K'(p_1,p_2)$ and $a_2\in K'(p_2,p_3)$, because then $K'[a_1,a_2]\cup\ell_1\cup\ell_2$ would be a pseudo-polygon with no concave articulation points, so the pseudoline segment $p_2p_i$ in $\cL$ would be disjoint from the exterior of it (and thus from the exterior of $K$) by Lemma~\ref{lem:forcing}, contradicting the assumption that $p_2p_i$ is not an edge of $G$.
\end{proof}

So, up to symmetry, we can assume that $a_1,a_2\in K'(p_1,p_2)$.
After possibly changing $p_i$, we can assume that there is no concave articulation point of $K'$ on $K'(p_3,p_i)$ with the same property.
By Lemma~\ref{lem:forcing}, since $p_2p_i$ is not an edge of $G$, the pseudo-polygon $K'[a_2,p_i]\cup\ell_2$ has a concave articulation point $p_j$ with $3\leq j<i$.
By Claim~\ref{claim:concave}, $p_3$ is not a concave articulation point of $K'$ so $4\leq j<i$.
We can assume that $j$ is maximal with that property, so that all articulation points of $K'$ on $K'(p_j,p_i)$ are convex.
Repeating the argument from the last claim and by the choice of $p_i$, there are two segments of $L_{j-1}$ and $L_j$ which have $p_j$ as one endpoint and have the other endpoint on $K(p_2,p_3)$; see Figure~\ref{fig:orderedHole} (bottom).
But then, there is a pseudo-polygon on $\cL$ with no concave articulation points that contains $p_2$ and $p_i$ (and $p_j$), and therefore, by Lemma~\ref{lem:forcing}, $p_2p_i$ is an edge of $G$.
This contradiction completes the proof.
\end{proof}

\section{Partitioning into capped graphs}
\label{sec:partition}

Recall that an ordered graph $(G,{\prec})$ is \emph{capped} if the following holds for any four vertices $a,b,c,d$ with $a\prec b\prec c\prec d$: if $ac\in E(G)$ and $bd\in E(G)$, then $ad\in E(G)$.
In contrast to previous notions defined in terms of $\prec$, this one is \emph{not} invariant under rotation of $\prec$.

\begin{lemma}
\label{lem:H-free-clique}
If\/ $u\prec v$ are two adjacent vertices of an\/ $\cH$-free ordered graph\/ $(G,{\prec})$ and\/ $X=\{u,v\}\cup\{x\in V(G)\colon u\prec x\prec v$ and\/ $ux,xv\in E(G)\}$, then\/ $(G[X],{\prec}|_X)$ is a capped graph.
\end{lemma}

\begin{proof}
If $a,b,c,d\in X$ and $ac,bd\in E(G)$, then $ac,bd$ is a crossing sequence from $a$ to $d$ and $du,va$ ($da$ if $a=u$ or $d=v$) is a crossing sequence from $d$ to $a$ in $(G,{\prec})$, so $ad\in E(G)$.
\end{proof}

The clique number of an $\cH$-free ordered graph $(G,{\prec})$ is the maximum clique number of the capped graph $(G[X],{\prec}|_X)$ defined in Lemma~\ref{lem:H-free-clique} over all choices of adjacent vertices $u\prec v$.
Thus, computing the clique number of $\cH$-free ordered graphs is reduced to computing the clique number of capped graphs, and the following proposition allows us to conclude Theorem~\ref{thm:H-free} from Theorem~\ref{thm:capped}.

\begin{proposition}
\label{prop:partition}
There is a polynomial-time algorithm that takes in an\/ $\cH$-free ordered graph\/ $(G,{\prec})$ and partitions its set of vertices into three subsets\/ $V_1$, $V_2$, and\/ $V_3$ so that for each\/ $i\in\{1,2,3\}$, the ordered graph\/ $(G[V_i],{\prec}|_{V_i})$ is capped.
\end{proposition}

Before delving into the full proof of Proposition~\ref{prop:partition}, we sketch the proof for the case that $(G,{\prec})$ is an ordered polygon visibility graph.
The sketch presents some key intuitions behind the full proof.
It is based on the ``window partition'' by Suri~\cite{Suri86}, which was used in a similar fashion to approximate chromatic variants of the well-known art gallery problem~\cite{bgmtw14, chromGuard19}.

\begin{proof}[Proof sketch for polygons]
We write $pq$ for the closed line segment connecting points $p$ and $q$.
Let $G=G(P,V)$ be a polygon visibility graph, where $P$ is a polygon with vertex set $V$.
Let $\prec$ be a natural ordering of $G$.
Let $x$ and $y$ be the smallest and the largest vertex in $\prec$, respectively, so that $xy$ is an edge of $P$ and of $G$.
Let $P_{xy}=(P\cup\int P)\setminus xy$, where $\int P$ is the interior of $P$.
We construct a partition of $V(G)$ into three sets, which we express in terms of a colouring $\phi$ of $V(G)$ that uses three colours: red, green, and blue.
First we describe a procedure that constructs a partition of $P_{xy}$ into ``windows''; these windows, as we will see, will be naturally arranged with a tree structure, and the root window will be ``based'' at $xy$.

To define the root window, we need to introduce the notion of ``visibility from $xy$''.
We say that a point $p\in P_{xy}$ is \emph{visible from\/ $xy$} if $p$ lies in the closed half-plane to the left of the line from $y$ to $x$ and there is a point $p'\in xy$ such that $pp'\subseteq P_{xy}\cup xy$.
The \emph{window\/ $W_{xy}$ based at\/ $xy$} consists of all points $p\in P_{xy}$ that are visible from $xy$.
It follows that $W_{xy}$ is a connected subset of $P_{xy}$; see Figure~\ref{fig:visibility_components} for an illustration.
This set $W_{xy}$ is the root of the constructed window partition tree.

\begin{figure}[t]
\centering
\begin{tikzpicture}[xscale=0.075, yscale=0.10, >=latex]
  \coordinate (y) at (0,0) {};
  \coordinate (x) at (19,0) {};
  \coordinate (ly) at (-2,-2) {};
  \coordinate (lx) at (21,-2) {};
  \coordinate (p1) at (23,4.9) {};
  \coordinate (lp1) at (22,6.4) {};
  \coordinate (p2) at (28,-4) {};
  \coordinate (p3) at (35,-4) {};
  \coordinate (p3p4) at (38.5,-0.9) {};
  \coordinate (p4) at (45,5) {};
  \coordinate (p5) at (37,14) {};
  \coordinate (p6) at (31,2) {};
  \coordinate (p6p7) at (29,6.2) {};
  \coordinate (lp6p7) at (30,7.5) {};
  \coordinate (p7) at (23,18) {};
  \coordinate (p8) at (46,20) {};
  \coordinate (p8p9) at (40.2,28.2) {};
  \coordinate (p9) at (37,33) {};
  \coordinate (p10) at (51,24) {};
  \coordinate (p11) at (53,34) {};
  \coordinate (p11p12) at (40.1,40) {};
  \coordinate (p12) at (36,42) {};
  \coordinate (p13) at (53,43.5) {};
  \coordinate (p13p14) at (37,48.5) {};
  \coordinate (p14) at (29,51) {};
  \coordinate (p14p15) at (24.1,47.9) {};
  \coordinate (p15) at (21,46) {};
  \coordinate (p16) at (29,42) {};
  \coordinate (p17) at (23,38) {};
  \coordinate (p18) at (33,24) {};
  \coordinate (p18p19_1) at (29,24) {};
  \coordinate (p18p19_2) at (15,24) {};
  \coordinate (p19) at (9,24) {};
  \coordinate (p20) at (7,32) {};
  \coordinate (p21) at (21,34) {};
  \coordinate (p22) at (17,42) {};
  \coordinate (p22p23) at (-0.4,37.8) {};
  \coordinate (p23) at (-4,37) {};
  \coordinate (p24) at (3,16) {};
  \coordinate (p25) at (17,12) {};
  \coordinate (p26) at (1,4.9) {};
  \coordinate (p27) at (-9,28) {};
  \coordinate (p27p28) at (-14.8,18.8) {};
  \coordinate (p28) at (-16,17) {};
  \coordinate (p29) at (-19,34) {};
  \coordinate (p29p30) at (-26,23) {};
  \coordinate (p30) at (-29,18) {};
  \coordinate (p30p31) at (-28.2,13.5) {};
  \coordinate (p31) at (-27,6) {};
  \coordinate (p32) at (-19,16) {};
  \coordinate (p33) at (-22,-1) {};
  \coordinate (p33p34) at (-18.7,0.5) {};
  \coordinate (p34) at (-11,4) {};
  \coordinate (p35) at (-17,8) {};
  \coordinate (p36) at (-5,10) {};
  \coordinate (p36p37) at (-11,0) {};
  \coordinate (p37) at (-16,-8) {};
  \coordinate (p38) at (-5,-2) {};
  \coordinate (p39) at (-9,-11) {};
  \coordinate (p40) at (5,-7) {};
  \draw[very thick] (y)--(x);
  \draw[fill=green!5, draw=none] (p1)--(p2)--(p3)--(p4)--(p5)--(p6)--(p6p7);
  \draw[green] (p1)--(p2)--(p3)--(p4)--(p5)--(p6)--(p6p7);
  \draw[fill=green!5, draw=none] (p7)--(p8)--(p9)--(p9)--(p10)--(p11)--(p12)--(p13)--(p14)--(p15)--(p16)--(p17)--(p18)--(p18p19_1);
  \draw[green] (p7)--(p8)--(p9)--(p9)--(p10)--(p11)--(p12)--(p13)--(p14)--(p15)--(p16)--(p17)--(p18)--(p18p19_1);
  \draw[fill=blue!5, draw=none] (p18p19_2)--(p19)--(p20)--(p21)--(p22)--(p23)--(p24)--(p25);
  \draw[blue] (p18p19_2)--(p19)--(p20)--(p21)--(p22)--(p23)--(p24)--(p25);
  \draw[fill=blue!5, draw=none] (p27p28)--(p28)--(p29)--(p30)--(p31)--(p32)--(p33)--(p34)--(p35)--(p36);
  \draw[blue] (p27p28)--(p28)--(p29)--(p30)--(p31)--(p32)--(p33)--(p34)--(p35)--(p36);
  \draw[fill=blue!5, draw=none] (p36p37)--(p37)--(p38)--(p39)--(p40)--(y);
  \draw[blue] (p36p37)--(p37)--(p38)--(p39)--(p40)--(y);
  \draw[fill=red!25, draw=none] (y)--(x)--(p1)--(p6p7)--(p7)--(p18p19_1)--(p18p19_2)--(p25)--(p26)--(p27)--(p27p28)--(p36)--(p36p37)--cycle;
  \draw[thick, very thick] (y)--(x);
  \draw[red] (x)--(p1);
  \draw[thick, red, very thick, <-] (p1)--(p6p7);
  \draw[red] (p6p7)--(p7);
  \draw[thick, red, very thick, <-] (p7)--(p18p19_1);
  \draw[red] (p18p19_1)--(p18p19_2);
  \draw[red, very thick, ->] (p18p19_2)--(p25);
  \draw[red] (p25)--(p26);
  \draw[red] (p26)--(p27);
  \draw[red] (p27)--(p27p28);
  \draw[red, very thick, ->] (p27p28)--(p36);
  \draw[red] (p36)--(p36p37);
  \draw[red, very thick, ->] (p36p37) -- (y);
  \draw[dashed] (p36p37) -- (y);
  \draw[dashed] (y) -- (p6p7);
  \draw[dashed] (x) -- (p18p19_2);
  \coordinate (z1) at (6.4,0) {};
  \draw[dashed] (z1) -- (p27p28);
  \coordinate (z2) at (5.5,0) {};
  \draw[dashed] (z2) -- (p18p19_1);
  \begin{footnotesize}
  \node[inner sep=2pt] at (lx) {$x$};
  \node[inner sep=2pt] at (ly) {$y$};
  \end{footnotesize}
  \begin{scriptsize}
  \node at (lp1) {$b$};
  \node at (lp6p7) {$a$};
  \end{scriptsize}
\end{tikzpicture}\hskip 1cm
\begin{tikzpicture}[xscale=0.075, yscale=0.10, >=latex]
  \coordinate (y) at (0,0) {};
  \coordinate (x) at (19,0) {};
  \coordinate (ly) at (-2,-2) {};
  \coordinate (lx) at (21,-2) {};
  \coordinate (p1) at (23,4.9) {};
  \coordinate (lp1) at (22,6.4) {};
  \coordinate (p2) at (28,-4) {};
  \coordinate (p3) at (35,-4) {};
  \coordinate (p3p4) at (38.5,-0.9) {};
  \coordinate (p4) at (45,5) {};
  \coordinate (p5) at (37,14) {};
  \coordinate (p6) at (31,2) {};
  \coordinate (p6p7) at (29,6.2) {};
  \coordinate (lp6p7) at (30,7.5) {};
  \coordinate (p7) at (23,18) {};
  \coordinate (p8) at (46,20) {};
  \coordinate (p8p9) at (40.2,28.2) {};
  \coordinate (p9) at (37,33) {};
  \coordinate (p10) at (51,24) {};
  \coordinate (p11) at (53,34) {};
  \coordinate (p11p12) at (40.1,40) {};
  \coordinate (p12) at (36,42) {};
  \coordinate (p13) at (53,43.5) {};
  \coordinate (p13p14) at (37,48.5) {};
  \coordinate (p14) at (29,51) {};
  \coordinate (p14p15) at (24.1,47.9) {};
  \coordinate (p15) at (21,46) {};
  \coordinate (p16) at (29,42) {};
  \coordinate (p17) at (23,38) {};
  \coordinate (p18) at (33,24) {};
  \coordinate (p18p19_1) at (29,24) {};
  \coordinate (p18p19_2) at (15,24) {};
  \coordinate (p19) at (9,24) {};
  \coordinate (p20) at (7,32) {};
  \coordinate (p21) at (21,34) {};
  \coordinate (p22) at (17,42) {};
  \coordinate (p22p23) at (-0.4,37.8) {};
  \coordinate (p23) at (-4,37) {};
  \coordinate (p24) at (3,16) {};
  \coordinate (p25) at (17,12) {};
  \coordinate (p26) at (1,4.9) {};
  \coordinate (p27) at (-9,28) {};
  \coordinate (p27p28) at (-14.8,18.8) {};
  \coordinate (p28) at (-16,17) {};
  \coordinate (p29) at (-19,34) {};
  \coordinate (p29p30) at (-26,23) {};
  \coordinate (p30) at (-29,18) {};
  \coordinate (p30p31) at (-28.2,13.5) {};
  \coordinate (p31) at (-27,6) {};
  \coordinate (p32) at (-19,16) {};
  \coordinate (p33) at (-22,-1) {};
  \coordinate (p33p34) at (-18.7,0.5) {};
  \coordinate (p34) at (-11,4) {};
  \coordinate (p35) at (-17,8) {};
  \coordinate (p36) at (-5,10) {};
  \coordinate (p36p37) at (-11,0) {};
  \coordinate (p37) at (-16,-8) {};
  \coordinate (p38) at (-5,-2) {};
  \coordinate (p39) at (-9,-11) {};
  \coordinate (p40) at (5,-7) {};
  \draw[very thick] (y)--(x);
  \draw[fill=red!30, draw=none] (p3p4)--(p4)--(p5)--(p6);
  \draw[red] (p3p4)--(p4)--(p5)--(p6);
  \draw[fill=green!30, draw=none] (p1)--(p2)--(p3)--(p3p4)--(p6)--(p6p7);
  \draw[green] (p1)--(p2)--(p3)--(p3p4)--(p6)--(p6p7);
  \draw[very thick,green,->] (p3p4)--(p6);
  \draw[fill=green!30, draw=none] (p9)--(p10)--(p11)--(p11p12);
  \draw[green] (p9)--(p10)--(p11)--(p11p12);
  \draw[fill=green!30, draw=none] (p12)--(p13)--(p13p14);
  \draw[green] (p12)--(p13)--(p13p14);
  \draw[fill=blue!30, draw=none] (p14p15)--(p15)--(p16);
  \draw[blue] (p14p15)--(p15)--(p16);
  \draw[fill=red!30, draw=none] (p8p9)--(p9)--(p11p12)--(p12)--(p13p14)--(p14)--(p14p15)--(p16)--(p17)--(p18);
  \draw[red] (p8p9)--(p9)--(p11p12)--(p12)--(p13p14)--(p14)--(p14p15)--(p16)--(p17)--(p18);
  \draw[very thick, red, <-] (p9)--(p11p12);
  \draw[very thick, red, <-] (p12)--(p13p14);
  \draw[very thick, red, ->] (p14p15)--(p16);
  \draw[fill=green!30, draw=none] (p7)--(p8)--(p8p9)--(p18)--(p18p19_1);
  \draw[green] (p7)--(p8)--(p8p9)--(p18)--(p18p19_1);
  \draw[very thick, green, ->] (p8p9)--(p18);
  \draw[fill=green!30, draw=none] (p19)--(p20)--(p21)--(p22)--(p22p23);
  \draw[green] (p19)--(p20)--(p21)--(p22)--(p22p23);
  \draw[fill=blue!30, draw=none] (p18p19_2)--(p19)--(p22p23)--(p23)--(p24)--(p25);
  \draw[blue] (p18p19_2)--(p19)--(p22p23)--(p23)--(p24)--(p25);
  \draw[blue, very thick, <-] (p19)--(p22p23);
  \draw[fill=green!30, draw=none] (p28)--(p29)--(p29p30);
  \draw[green] (p28)--(p29)--(p29p30);
  \draw[fill=red!30, draw=none] (p30p31)--(p31)--(p32);
  \draw[red] (p30p31)--(p31)--(p32);
  \draw[fill=red!30, draw=none] (p33p34)--(p34)--(p35);
  \draw[red] (p33p34)--(p34)--(p35);
  \draw[fill=blue!30, draw=none] (p27p28)--(p28)--(p29p30)--(p30)--(p30p31)--(p32)--(p33)--(p33p34)--(p35)--(p36);
  \draw[blue] (p27p28)--(p28)--(p29p30)--(p30)--(p30p31)--(p32)--(p33)--(p33p34)--(p35)--(p36);
  \draw[very thick, blue, <-] (p28)--(p29p30);
  \draw[very thick, blue, ->] (p30p31)--(p32);
  \draw[very thick, blue, ->] (p33p34)--(p35);
  \draw[fill=blue!30, draw=none] (p36p37)--(p37)--(p38)--(p39)--(p40)--(y);
  \draw[blue] (p36p37)--(p37)--(p38)--(p39)--(p40)--(y);
  \draw[fill=red!30, draw=none] (x)--(p1)--(p6p7)--(p7)--(p18p19_1)--(p18p19_2)--(p25)--(p26)--(p27)--(p27p28)--(p36)--(p36p37)--(y);
  \draw[red] (x)--(p1)--(p6p7)--(p7)--(p18p19_1)--(p18p19_2)--(p25)--(p26)--(p27)--(p27p28)--(p36)--(p36p37)--(y);
  \draw[thick,very thick] (y)--(x);
  \draw[red,very thick, <-] (p1)--(p6p7);
  \draw[red,very thick, <-] (p7)--(p18p19_1);
  \draw[red,very thick, ->] (p18p19_2)--(p25);
  \draw[red,very thick, ->] (p27p28)--(p36);
  \draw[red,very thick, ->] (p36p37) -- (y);
  \begin{footnotesize}
  \node[inner sep=2pt] at (lx) {$x$};
  \node[inner sep=2pt] at (ly) {$y$};
  \end{footnotesize}
  \tikzstyle{every node}=[circle, minimum size=3.5pt, inner sep=0pt, draw, fill=blue]
  \node at (p15) {};
  \node at (p19) {};
  \node at (p23) {};
  \node at (p24) {};
  \node at (p28) {};
  \node at (p30) {};
  \node at (p32) {};
  \node at (p33) {};
  \node at (p35) {};
  \node at (p37) {};
  \node at (p38) {};
  \node at (p39) {};
  \node at (p40) {};
  \tikzstyle{every node}=[circle, minimum size=3.5pt, inner sep=0pt, draw, fill=green]
  \node at (p2) {};
  \node at (p3) {};
  \node at (p6) {};
  \node at (p8) {};
  \node at (p10) {};
  \node at (p11) {};
  \node at (p13) {};
  \node at (p18) {};
  \node at (p20) {};
  \node at (p21) {};
  \node at (p22) {};
  \node at (p29) {};
  \tikzstyle{every node}=[circle, minimum size=3.5pt, inner sep=0pt, draw, fill=red]
  \node at (x) {};
  \node at (y) {};
  \node at (p1) {};
  \node at (p7) {};
  \node at (p25) {};
  \node at (p26) {};
  \node at (p27) {};
  \node at (p36) {};
  \node at (p4) {};
  \node at (p5) {};
  \node at (p9) {};
  \node at (p12) {};
  \node at (p14) {};
  \node at (p16) {};
  \node at (p17) {};
  \node at (p31) {};
  \node at (p34) {};
\end{tikzpicture}
\caption{To the left: a polygon $P$ with window $W_{xy}$ based at $xy$ in red, oriented lines $\protect\overrightarrow{L_{ab}}$ depicted with red arrows and dashed lines, and the left/right-invisible sets $I_{ab}$ in green/blue, respectively.
To the right: the final window partition of $P_{ab}$.}
\label{fig:visibility_components}
\end{figure}
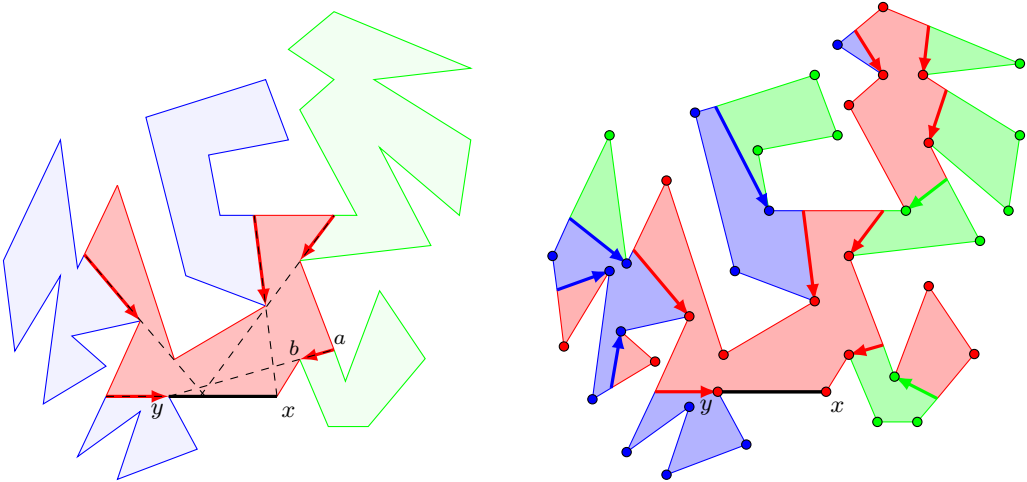

The points in $P_{xy}\setminus W_{xy}$ form some number (possibly zero) of connected subsets of $P_{xy}$.
It can be shown that each such set is of the form $I_{ab}$ for some polygon $I$ and edge $ab$ of $I$, where $a$ and $b$ are on $P$ and every point in the segment $ab$ is visible from $xy$ in $P$.
Furthermore, at least one of the points $a$ and $b$ is a vertex of $P$ and the line $L_{ab}$ going through $a$ and $b$ intersects $xy$; we direct $L_{ab}$ from $a$ and $b$ towards this intersection point to obtain an oriented line $\overrightarrow{L_{ab}}$.
Given this description, we partition the invisible sets $I_{ab}$ into two groups:
\begin{itemize}
\item $I_{ab}$ is \emph{left-invisible} if it is ``towards the left side'' of $\overrightarrow{L_{ab}}$;
\item $I_{ab}$ is \emph{right-invisible} if it is ``towards the right side'' of $\overrightarrow{L_{ab}}$.
\end{itemize}
We do not give formal definitions, but refer to Figure~\ref{fig:visibility_components}.

Given this partition, it can be shown that there are no mutually visible points in two different left-invisible sets or two different right-invisible sets.
Now, for each invisible set $I_{ab}$, we can recursively obtain a window partition of $I_{ab}$ which is rooted at a window $W_{ab}$ based at $ab$.
The window partition of $P_{xy}$ is then obtained by making each of these windows $W_{ab}$ a \emph{left-child} or a \emph{right-child} of $W_{xy}$ according to whether $I_{ab}$ is left-invisible or right-invisible.
The following observation summarizes this construction of the window partition: if two points in different windows $W_1$ and $W_2$ are mutually visible, then either $W_1$ and $W_2$ are in a parent-child relationship, or there is a window $W$ such that one of $W_1$, $W_2$ is a left-child of $W$ and the other is a right-child of $W$.

Now we show how to obtain the $3$-colouring $\phi$ of the vertex set of $(G,{\prec})$ such that each colour class induces a capped subgraph.
The property above allows us to colour the windows by three colours (say, red, green, and blue) so that no two points in two different windows of the same colour are mutually visible.
We colour the root window, say, by red.
Then we extend this colouring on the remaining windows so that the children of each window $W$ obtain a colour different from $W$ and the left children of $W$ are coloured with a different colour from the right children of $W$.
This way every vertex of $P$ other than $x$ and $y$ is coloured.
We colour $x$ and $y$ arbitrarily; see Figure~\ref{fig:visibility_components} (right).

To complete the proof, we need to show that the vertices of $P$ in $W_{xy}\cup\{x,y\}$ induce a capped subgraph of $(G,{\prec})$; for the other windows we can apply induction.
It is well known that the related class of ordered terrain visibility graphs is capped \cite[Lemma~1]{Evans2015}, but we give a proof sketch anyway, because that lemma does not apply directly.

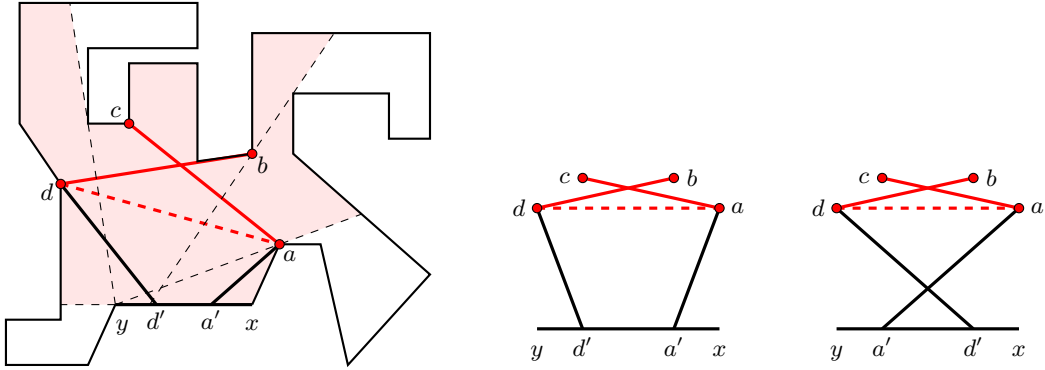
\begin{figure}[t]
\centering
\begin{tikzpicture}[xscale=0.18, yscale=0.2]
  \coordinate (y) at (0,0) {};
  \coordinate (x) at (10,0) {};
  \coordinate (ly) at (0.5,-1.5) {};
  \coordinate (lx) at (10,-1.35) {};
  \coordinate (p1) at (12,4) {};
  \coordinate (a) at (12,4) {};
  \coordinate (la) at (12.75,3.25) {};
  \coordinate (p2) at (15,4) {};
  \coordinate (p3) at (17,-4) {};
  \coordinate (p4) at (23,2) {};
  \coordinate (p4p5) at (18,6) {};
  \coordinate (p5) at (13,10) {};
  \coordinate (p5p6) at (13,13) {};
  \coordinate (p6) at (13,14) {};
  \coordinate (p7) at (20,14) {};
  \coordinate (p8) at (20,11) {};
  \coordinate (p9) at (23,11) {};
  \coordinate (p10) at (23,18) {};
  \coordinate (p10p11) at (16,18) {};
  \coordinate (sp10p11) at (2.5,0) {};
  \coordinate (p11) at (10,18) {};
  \coordinate (p12) at (10,10) {};
  \coordinate (b) at (10,10) {};
  \coordinate (lb) at (10.75,9.25) {};
  \coordinate (p13) at (6,9.5) {};
  \coordinate (p14) at (6,16) {};
  \coordinate (p15) at (1,16) {};
  \coordinate (p16) at (1,12) {};
  \coordinate (c) at (1,12) {};
  \coordinate (lc) at (0,12.75) {};
  \coordinate (p17) at (-2,12) {};
  \coordinate (p18) at (-2,17) {};
  \coordinate (p19) at (6,17) {};
  \coordinate (p20) at (6,20) {};
  \coordinate (p20p21) at (-3.33,20) {};
  \coordinate (p21) at (-7,20) {};
  \coordinate (p22) at (-7,12) {};
  \coordinate (p23) at (-4,8) {};
  \coordinate (d) at (-4,8) {};
  \coordinate (ld) at (-5,7.5) {};
  \coordinate (p23p24) at (-4,0) {};
  \coordinate (p24) at (-4,-1) {};
  \coordinate (p25) at (-8,-1) {};
  \coordinate (p26) at (-8,-4) {};
  \coordinate (p27) at (-2,-4) {};
  \coordinate (d') at (3,0) {};
  \coordinate (a') at (7,0) {};
  \coordinate (ld') at (3,-1) {};
  \coordinate (la') at (7,-1) {};
  \draw[fill=red!10, draw=none] (x)--(p1)--(p4p5)--(p5)--(p6)--(p10p11)--(p11)--(p12)--(p13)--(p14)--(p15)--(p16)--(p17)--(p20p21)--(p21)--(p22)--(p23)--(p23p24)--(y);
  \draw[very thick] (y) -- (x);
  \draw[very thick, red] (a) -- (c);
  \draw[very thick, red] (b) -- (d);
  \draw[red, very thick, dashed] (d) -- (a);
  \draw[thick] (x)--(p1)--(p2)--(p3)--(p4)--(p5)--(p6)--(p7)--(p8)--(p9)--(p10)--(p11)--(p12)--(p13)--(p14)--(p15)--(p16)--(p17)--(p18)--(p19)--(p20)--(p21)--(p22)--(p23)--(p24)--(p25)--(p26)--(p27)--(y);
  \draw[very thick] (d) -- (d') {};
  \draw[very thick] (a) -- (a') {};
  \draw[dashed, thin] (y) -- (p4p5);
  \draw[dashed, thin] (sp10p11) -- (p10p11) {};
  \draw[dashed, thin] (y) -- (p20p21) {};
  \draw[dashed, thin] (y) -- (p23p24) {};
  \begin{footnotesize}
  \tikzstyle{every node}=[inner sep=2pt]
  \node at (la) {$a$};
  \node at (lb) {$b$};
  \node at (lc) {$c$};
  \node at (ld) {$d$};
  \node at (lx) {$x$};
  \node at (ly) {$y$};
  \node at (ld') {$d'$};
  \node at (la') {$a'$};
  \end{footnotesize}
  \tikzstyle{every node}=[circle, minimum size=3.5pt, inner sep=0pt, draw, fill=red]
  \node at (a) {};
  \node at (b) {};
  \node at (c) {};
  \node at (d) {};
\end{tikzpicture}\hskip 1cm
\begin{tikzpicture}[xscale=0.6, yscale=0.4]
  \coordinate (y) at (0,0) {};
  \coordinate (x) at (4,0) {};
  \coordinate (lu) at (0,4) {};
  \coordinate (lb) at (1,0) {};
  \coordinate (ru) at (4,4) {};
  \coordinate (rb) at (3,0) {};
  \coordinate (m) at (2,2) {};
  \coordinate (lm) at (1.5,2) {};
  \coordinate (q) at (2,4.65) {};
  \coordinate (lq) at (2,5.25) {};
  \coordinate (c) at (1,5) {};
  \coordinate (lc) at (0.6,5) {};
  \coordinate (be) at (3,5) {};
  \coordinate (lbe) at (3.4,5) {};
  \coordinate (ly) at (0,-1.25) {};
  \coordinate (lx) at (4,-1.1) {};
  \coordinate (llu) at (-0.4,4) {};
  \coordinate (llb) at (1,-1.1) {};
  \coordinate (lru) at (4.4,4) {};
  \coordinate (lrb) at (3,-1.1) {};
  \draw[very thick] (y) -- (x);
  \draw[very thick] (lb) -- (lu);
  \draw[very thick] (rb) -- (ru);
  \draw[very thick, red] (lu) -- (be);
  \draw[very thick, red] (ru) -- (c);
  \draw[very thick, red, dashed] (lu) -- (ru);
  \begin{footnotesize}
  \tikzstyle{every node}=[inner sep=2pt]
  \node[above] at (lx) {$x$};
  \node[above] at (ly) {$y$};
  \node at (llu) {$d$};
  \node[above] at (llb) {$d'$};
  \node at (lru) {$a$};
  \node[above] at (lrb) {$a'$};
  \node at (lc) {$c$};
  \node at (lbe) {$b$};
  \end{footnotesize}
  \tikzstyle{every node}=[circle, minimum size=3.5pt, inner sep=0pt, draw, fill=red]
  \node at (lu) {};
  \node at (ru) {};
  \node at (c) {};
  \node at (be) {};
\end{tikzpicture}\hskip .75cm
\begin{tikzpicture}[xscale=0.6, yscale=0.4]
  \coordinate (y) at (0,0) {};
  \coordinate (x) at (4,0) {};
  \coordinate (lu) at (0,4) {};
  \coordinate (lb) at (1,0) {};
  \coordinate (ru) at (4,4) {};
  \coordinate (rb) at (3,0) {};
  \coordinate (m) at (2,2) {};
  \coordinate (lm) at (1.5,2) {};
  \coordinate (q) at (2,4.65) {};
  \coordinate (lq) at (2,5.25) {};
  \coordinate (c) at (1,5) {};
  \coordinate (lc) at (0.6,5) {};
  \coordinate (be) at (3,5) {};
  \coordinate (lbe) at (3.4,5) {};
  \coordinate (ly) at (0,-1.25) {};
  \coordinate (lx) at (4,-1.1) {};
  \coordinate (llu) at (-0.4,4) {};
  \coordinate (llb) at (1,-1.1) {};
  \coordinate (lru) at (4.4,4) {};
  \coordinate (lrb) at (3,-1.1) {};
  \draw[very thick] (y) -- (x);
  \draw[very thick] (lb) -- (ru);
  \draw[very thick] (rb) -- (lu);
  \draw[very thick, red] (lu) -- (be);
  \draw[very thick, red] (ru) -- (c);
  \draw[very thick, red, dashed] (lu) -- (ru);
  \begin{footnotesize}
  \tikzstyle{every node}=[inner sep=2pt]
  \node[above] at (lx) {$x$};
  \node[above] at (ly) {$y$};
  \node at (llu) {$d$};
  \node[above] at (llb) {$a'$};
  \node at (lru) {$a$};
  \node[above] at (lrb) {$d'$};
  \node at (lc) {$c$};
  \node at (lbe) {$b$};
  \end{footnotesize}
  \tikzstyle{every node}=[circle, minimum size=3.5pt, inner sep=0pt, draw, fill=red]
  \node at (lu) {};
  \node at (ru) {};
  \node at (c) {};
  \node at (be) {};
\end{tikzpicture}
\caption{Possible relations between the segments $a'a,d'd,ac,bd,yx$.}
\label{fig:abcd}
\end{figure}

Suppose there are four vertices $a,b,c,d\in W_{xy}\cup\{x,y\}$ such that $a\prec b\prec c\prec d$, $ac\in E(G)$, and $bd\in E(G)$ (it is possible that $a=x$ and/or $d=y$).
By the definition of visibility from $xy$, all four points $a,b,c,d$ are in the closed half-plane to the left of the line from $y$ to $x$.
Let $a',d'\in xy$ be such that the segments $a'a$ and $d'd$ are disjoint from the exterior of $P$; see Figure~\ref{fig:abcd} for an illustration.
Now, the five segments $a'a,d'd,ac,bd,yx$ divide the plane into a set $\cF$ of faces, and exactly one of the faces in $\cF$ is unbounded (the outer face).
Now, to show that $ad$ is disjoint from the exterior of $P$, it suffices to prove the following two claims.
\begin{itemize}
\item The polygon $P$ is contained in the closure of the outer face of $\cF$ and $P$ does not cross (but may touch) any segment in the set $\{a'a,d'd,ac,bd\}$.
\item The segment $ad$ is disjoint from the interior of the outer face of $\cF$.
\end{itemize}
The first claim is quite obvious; see the left side of Figure~\ref{fig:abcd} for an illustration.
The second claim can be proven by considering all the cases for how the segments $a'a$, $d'd$, $ac$, and $bd$ can be placed with respect to each other.
We leave the details to the reader.
\end{proof}

Now, we proceed to the proof of Proposition~\ref{prop:partition} in the general case.
It is convenient to word it in terms of colourings; we will find a $3$-colouring such that every colour class induces a capped subgraph.
The proof will work by extending certain partial colourings ``inside a valid segment''.
To explain this, we need to introduce some notation.

Let $(G,{\prec})$ be an ordered graph.
A \emph{segment} of $(G,{\prec})$ is a set of at least two consecutive vertices; that is, a segment is a set $V[x,y]$ consisting of vertices $x$ and $y$ with $x\prec y$, and all vertices $v$ with $x\prec v\prec y$.
Such vertices $v$ are called \emph{interior}, the vertices $x$ and $y$ are called the \emph{ends}, and vertices in $V(G)\setminus V[x,y]$ are called \emph{exterior}.
We write $V(x,y)$ for $V[x,y]\setminus\{x,y\}$ and similarly for $V[x,y)$ and $V(x,y]$.
A segment is \emph{valid} if there is a crossing sequence from $y$ to $x$ (see Figure~\ref{fig:validSeg}).

\begin{figure}[t]
\centering
\begin{tikzpicture}[scale=.7, every node/.style={inner sep=2, outer sep=0}, vtx/.style={draw, circle, fill=gray}]
  \node[vtx, inner sep=3, label=below left:$x$] (x) at (-4,0) {};
  \node[vtx, inner sep=3, label=below right:$y$] (y) at (4,0) {};
  \node[label=below:{$V[x,y]$}] at (0,-.85) {};
  \draw[ultra thick] (x) -- (-4,-.85) -- (4,-.85) -- (y);
  \node[vtx] (a1) at (6.5,0) {};
  \node[vtx] (a2) at (8,0) {};
  \node[vtx] (a3) at (-8.5,0) {};
  \node[vtx] (a4) at (-7,0) {};
  \draw[thick] (a4) to [bend left=30] (a1);
  \draw[thick] (a3) to [bend left=30] (x);
  \draw[thick] (y) to [bend left=30] (a2);
\end{tikzpicture}
\caption{A valid segment $V[x,y]$.}
\label{fig:validSeg}
\end{figure}
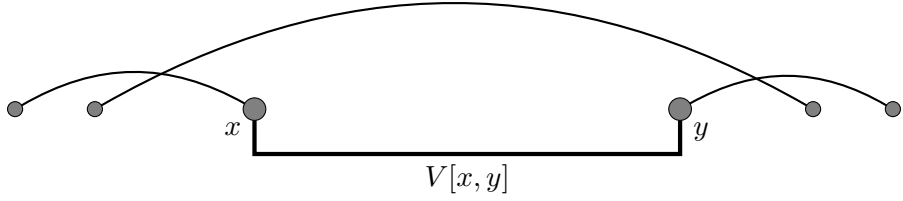

Recall the proof sketch for partitioning the vertex set of a polygon visibility graph into three capped graphs.
We had a partial colouring $\phi$ and a segment $V[x,y]$ such that $x$ and $y$ were coloured and connected by an edge.
The coloured interior vertices were exactly those vertices which were ``visible'' to the edge $xy$, and these vertices were all given the same colour.
We then partitioned the interior vertices into segments which were ``left-invisible'' and segments which were ``right-invisible'' and continued the process.

Now, instead of requiring that $x$ and $y$ are connected by an edge, we just require that the segment $V[x,y]$ is valid.
It turns out that even for visibility graphs of points on a Jordan curve, it is not always possible to tell from $(G,{\prec})$ alone if an ``invisible'' vertex is ``left-invisible'' or ``right-invisible''; the vertex could, for instance, even be isolated.
So instead we will define left-reachable vertices which definitely cannot be ``left-invisible'' and right-reachable vertices which definitely cannot be ``right-invisible''.
Vertices which are both left- and right-reachable will be called strongly reachable; these are the vertices that are ``visible''.
The proof ends up being rather technical, but these definitions are key.
For the following, refer to Figure~\ref{fig:reachable}.

\begin{figure}[t]
\centering
\begin{tikzpicture}[scale=.7, every node/.style={inner sep=2, outer sep=0}, vtx/.style={draw, circle, fill=gray}]
  \node[vtx, inner sep=3, label=below left:$x$] (x) at (-4,0) {};
  \node[vtx, inner sep=3, label=below right:$y$] (y) at (4,0) {};
  \node[label=below:{$V[x,y]$}] at (0,-.85) {};
  \draw[ultra thick] (x) -- (-4,-.85) -- (4,-.85) -- (y);
  \node[vtx] (a1) at (6.5,0) {};
  \node[vtx] (a2) at (8,0) {};
  \node[vtx] (a3) at (-8.5,0) {};
  \node[vtx] (a4) at (-7,0) {};
  \draw[thick] (a4) to [bend left=30] (a1);
  \draw[thick] (a3) to [bend left=30] (x);
  \draw[thick] (y) to [bend left=30] (a2);
  \node[vtx, label=below:$S$] (vS) at (-2.5,0) {};
  \node[vtx, label=below:$R$] (vR) at (-1,0) {};
  \node[vtx, label=below:$L$] (vL) at (2.5,0) {};
  \node[vtx] (b1) at (-5.5,0) {};
  \node[vtx, label=below:$S$] (b2) at (0,0) {};
  \node[vtx, label=below:$L$] (b3) at (1,0) {};
  \node[vtx] (b4) at (5.25,0) {};
  \draw[thick] (b1) to [bend left=30] (vS);
  \draw[thick] (vL) to [bend left=30] (y);
  \draw[thick] (vR) to [bend left=30] (b3);
  \draw[thick] (b2) to [bend left=30] (b4);
\end{tikzpicture}
\caption{Internal vertices which are strongly-reachable (labelled $S$) or left- or right-reachable only (labelled $L$ and $R$ respectively).}
\label{fig:reachable}
\end{figure}
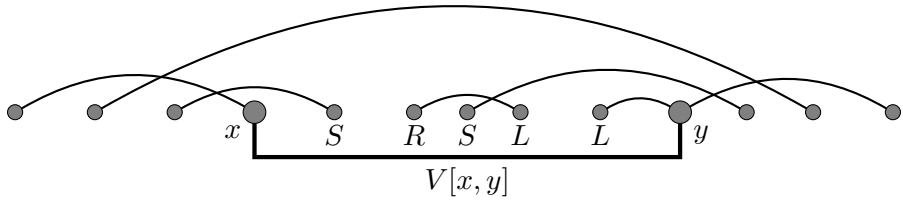

Let $V[x,y]$ be a valid segment.
A vertex $v$ is \emph{left-reachable} from $V[x,y]$ if it is in $V[x,y)$ and there is a crossing sequence from $y$ to $v$ (so in particular $x$ is left-reachable).
It is \emph{right-reachable} if it is in $V(x,y]$ and there is a crossing sequence from $v$ to $x$ (so $y$ is right-reachable).
It is \emph{strongly-reachable} if it is both left- and right-reachable (so all strongly reachable vertices are interior).
We write $L[x,y)$, $R(x,y]$, and $S(x,y)$ for the set of left-, right-, and strongly-reachable vertices, respectively.
So $S(x,y) = L[x,y)\cap R(x,y]$.
First we observe the following.

\begin{lemma}
\label{lem:extInt}
For any ordered graph\/ $(G,{\prec})$ and valid segment\/ $V[x,y]$, every interior vertex which is adjacent to an exterior vertex is strongly reachable.
\end{lemma}

\begin{proof}
Let $v$ be an interior vertex which is adjacent to an exterior vertex $u$.
There are crossing sequences from $y$ to $x$ and from $u$ to $v$.
Thus, by Lemma~\ref{lem:almostTran}, there is a crossing sequence from $y$ to $v$, and so $v$ is left-reachable.
Likewise, since there are crossing sequences from $v$ to $u$ and from $y$ to $x$, the vertex $v$ is also right-reachable, so it is strongly-reachable.
\end{proof}

As further motivation for these definitions, we prove the following.

\begin{lemma}
\label{lem:induceCapped}
For any\/ $\cH$-free ordered graph\/ $(G,{\prec})$ and valid segment\/ $V[x,y]$, the set\/ $S(x,y)\cup\{x,y\}$ induces a capped subgraph.
\end{lemma}

\begin{proof}
Otherwise, there are vertices $a,b,c,d\in S(x,y)\cup\{x,y\}$ such that $a\prec b\prec c\prec d$, and $ac,bd\in E(G)$, yet $ad\notin E(G)$.
Then, since $d$ is right-reachable, there is a crossing sequence from $d$ to $x$.
Since $a$ is left-reachable, there is a crossing sequence from $y$ to $a$.
Thus, there is a crossing sequence from $d$ to $a$ (this is trivial if $d=y$ or $a=x$, and otherwise we apply Lemma~\ref{lem:almostTran}).
This contradicts the assumption that $(G,{\prec})$ is $\cH$-free.
\end{proof}

Now we need to define certain partial colourings and what it means to extend them.
For the following definition, refer to Figure~\ref{fig:partial}.

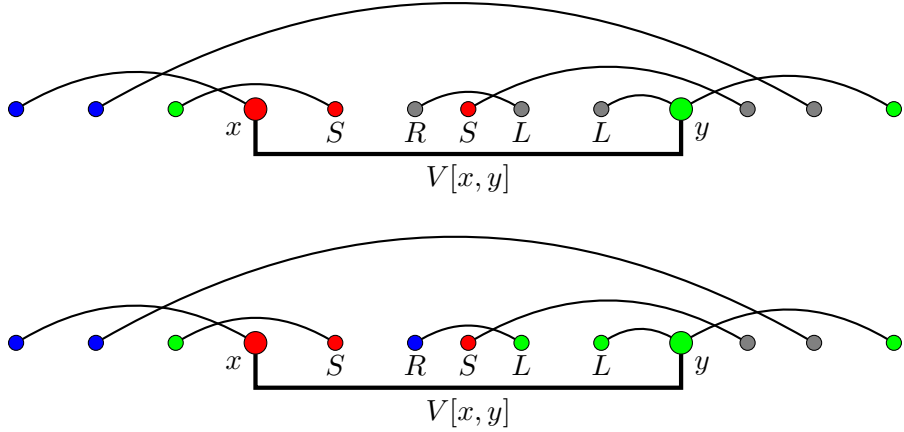
\begin{figure}[t]
\centering
\begin{tikzpicture}[scale=.7, every node/.style={inner sep=2, outer sep=0}, vtx/.style={draw, circle, fill=gray}]
  \node[vtx, inner sep=3, label=below left:$x$, fill=red] (x) at (-4,0) {};
  \node[vtx, inner sep=3, label=below right:$y$, fill=green] (y) at (4,0) {};
  \node[label=below:{$V[x,y]$}] at (0,-.85) {};
  \draw[ultra thick] (x) -- (-4,-.85) -- (4,-.85) -- (y);
  \node[vtx] (a1) at (6.5,0) {};
  \node[vtx, fill=green] (a2) at (8,0) {};
  \node[vtx, fill=blue] (a3) at (-8.5,0) {};
  \node[vtx, fill=blue] (a4) at (-7,0) {};
  \draw[thick] (a4) to [bend left=30] (a1);
  \draw[thick] (a3) to [bend left=30] (x);
  \draw[thick] (y) to [bend left=30] (a2);
  \node[vtx, label=below:$S$, fill=red] (vS) at (-2.5,0) {};
  \node[vtx, label=below:$R$] (vR) at (-1,0) {};
  \node[vtx, label=below:$L$] (vL) at (2.5,0) {};
  \node[vtx, fill=green] (b1) at (-5.5,0) {};
  \node[vtx, label=below:$S$, fill=red] (b2) at (0,0) {};
  \node[vtx, label=below:$L$] (b3) at (1,0) {};
  \node[vtx] (b4) at (5.25,0) {};
  \draw[thick] (b1) to [bend left=30] (vS);
  \draw[thick] (vL) to [bend left=30] (y);
  \draw[thick] (vR) to [bend left=30] (b3);
  \draw[thick] (b2) to [bend left=30] (b4);
\end{tikzpicture}
\begin{tikzpicture}[scale=.7, every node/.style={inner sep=2, outer sep=0}, vtx/.style={draw, circle, fill=gray}]
  \node[vtx, inner sep=3, label=below left:$x$, fill=red] (x) at (-4,0) {};
  \node[vtx, inner sep=3, label=below right:$y$, fill=green] (y) at (4,0) {};
  \node[label=below:{$V[x,y]$}] at (0,-.85) {};
  \draw[ultra thick] (x) -- (-4,-.85) -- (4,-.85) -- (y);
  \node[vtx] (a1) at (6.5,0) {};
  \node[vtx, fill=green] (a2) at (8,0) {};
  \node[vtx, fill=blue] (a3) at (-8.5,0) {};
  \node[vtx, fill=blue] (a4) at (-7,0) {};
  \draw[thick] (a4) to [bend left=30] (a1);
  \draw[thick] (a3) to [bend left=30] (x);
  \draw[thick] (y) to [bend left=30] (a2);
  \node[vtx, label=below:$S$, fill=red] (vS) at (-2.5,0) {};
  \node[vtx, label=below:$R$, fill=blue] (vR) at (-1,0) {};
  \node[vtx, label=below:$L$, fill=green] (vL) at (2.5,0) {};
  \node[vtx, fill=green] (b1) at (-5.5, 0) {};
  \node[vtx, label=below:$S$, fill=red] (b2) at (0,0) {};
  \node[vtx, label=below:$L$, fill=green] (b3) at (1,0) {};
  \node[vtx] (b4) at (5.25,0) {};
  \draw[thick] (b1) to [bend left=30] (vS);
  \draw[thick] (vL) to [bend left=30] (y);
  \draw[thick] (vR) to [bend left=30] (b3);
  \draw[thick] (b2) to [bend left=30] (b4);
\end{tikzpicture}
\caption{A $V[x,y]$-precolouring $\phi$ (top) and a $V[x,y]$-extension of $\phi$ (bottom).}
\label{fig:partial}
\end{figure}

For an ordered graph $(G,\prec)$ and a valid segment $V[x,y]$, a \emph{$V[x,y]$-precolouring} is a partial $3$-colouring $\phi$ such that
\begin{enumeratei}
\item every colour class of $\phi$ induces a capped subgraph,
\item $\phi$ colours $S(x,y)\cup\{x,y\}$ and colours no other interior vertices, and
\item $\phi$ colours every vertex in $S(x,y)$ the same colour, and no exterior vertex which is a neighbour of an interior vertex is given this colour.
\end{enumeratei}
Finally, $\phi$ \emph{extends inside} $V[x,y]$ if the uncoloured interior vertices can be coloured so that for each colour class $X$, the set $X\cap V[x,y]$ induces a capped subgraph.
The new colouring is a \emph{$V[x,y]$-extension} of $\phi$.
The next lemma implies that each colour class of a $V[x,y]$-extension induces a capped subgraph.

\begin{lemma}
\label{lem:extends}
If\/ $(G,{\prec})$ is an ordered graph with a segment\/ $V[x,y]$ such that no interior vertex is adjacent to an exterior vertex and each of the sets\/ $V[x,y]$ and\/ $V(G)\setminus V(x,y)$ induces a capped subgraph, then\/ $(G,{\prec})$ is capped.
\end{lemma}

\begin{proof}
Suppose for the sake of contradiction that there are vertices $a,b,c,d$ such that $a\prec b\prec c\prec d$ and $ac,bd\in E(G)$ while $ad\notin E(G)$.
Then at least one of $a,b,c,d$ must be exterior to $V[x,y]$ and at least one must be interior.
If $b$ is exterior, then either $a$ or $d$ is also exterior.
So, up to symmetry, we can assume that $a$ is exterior.
Then $a\prec x$ and $c$ is not interior.
It follows that exactly one of $b$ and $d$ is interior and the other is exterior---a contradiction.
\end{proof}

We will use the following lemma in order to find valid segments inside a valid segment $V[x,y]$.
Conditions~(ii) and~(iii) are symmetric.
As an example, note that in Figure~\ref{fig:partial}, condition~(i) implies that the second and third vertices in $V[x,y]$ form a valid segment, and condition~(iii) implies that the first and second vertices in $V[x,y]$ form a valid segment.

\begin{lemma}
\label{lem:valid}
For any ordered graph\/ $(G,{\prec})$ and valid segment\/ $V[x,y]$, each segment\/ $V[a,b]$ which satisfies one of the following conditions is valid:
\begin{enumeratei}
\item $a\in L[x,y)$ and\/ $b\in R(x,y]$;
\item $a,b\in L[x,y)$, no vertex in\/ $V(a,b)$ is left-reachable from\/ $V[x,y]$, and there is a vertex in\/ $V(b,y)$ which is right-reachable from\/ $V[x,y]$;
\item $a,b\in R(x,y]$, no vertex in\/ $V(a,b)$ is right-reachable from\/ $V[x,y]$, and there is a vertex in\/ $V(x,a)$ which is left-reachable from\/ $V[x,y]$.
\end{enumeratei}
\end{lemma}

\begin{proof}
First suppose that condition (i) holds.
Then there are crossing sequences from $b$ to $x$ and from $y$ to $a$ and thus, using Lemma~\ref{lem:almostTran} if $b\neq y$ and $a\neq x$, also from $b$ to $a$.
So $V[a,b]$ is valid.

Conditions (ii) and (iii) are symmetric via reversing $\prec$.
So it suffices to consider the case that condition (ii) holds.
Then there are crossing sequences from $y$ to $a$ and from $y$ to $b$.
If the crossing sequence from $y$ to $b$ is a single edge, then, since there is a vertex in $V(b,y)$ which is right-reachable and by Lemma~\ref{lem:almostTran}, there is a crossing sequence from $b$ to $x$.
Then $b$ is right-reachable and $V[a,b]$ is valid by condition (i).

So there is a crossing sequence, say $e_1,\ldots,e_k$, from $y$ to $b$ with at least two edges.
Let $v$ be the end of $e_{k-1}$ such that $e_1,\ldots,e_{k-1}$ is a crossing sequence from $y$ to $v$ (so if $k=2$, then $v$ is the end of $e_{k-1}=e_1$ which is not $y$).
Then $v\notin V(a,b)$ since then $v$ would be left-reachable.
This implies that the end $u$ of $e_k$ which is not $b$ satisfies $u\notin V[a,y]$.
Thus, since $e_k$ is a crossing sequence from $b$ to $u$ and there is a crossing sequence from $y$ to $a$, by Lemma~\ref{lem:almostTran} there is a crossing sequence from $b$ to $a$.
So $V[a,b]$ is valid.
\end{proof}

We are ready to begin the proof of the main proposition.
Afterwards, we will quickly show how to apply it when there is no fixed precolouring.

\begin{proposition}
\label{prop:extend}
There is a polynomial-time algorithm which takes in an\/ $\cH$-free ordered graph\/ $(G,{\prec})$, vertices\/ $x$ and\/ $y$ such that\/ $V[x,y]$ is a valid segment, and a\/ $V[x,y]$-precolouring\/ $\phi$, and returns a\/ $V[x,y]$-extension of\/ $\phi$.
\end{proposition}

\begin{proof}
Throughout the proof we will implicitly use the fact that there is a polynomial-time algorithm to determine whether there is a crossing sequence from a vertex $u$ to a vertex $v$; see the proof of Proposition~\ref{prop:H-free-alg}.
So in particular we can find the sets $L[x,y)$, $R(x,y]$, and $S(x,y)$.
If all interior vertices are coloured, we just return $\phi$.
So suppose that there is an uncoloured interior vertex.
We split into cases and take care of the special case first.

\emph{Case 1.}\enspace
The set $S(x,y)$ is empty.

For the following definitions, it is convenient to add an edge between any pair of consecutive, non-adjacent vertices in $V[x,y]$; that is, if $a$ and $b$ are non-adjacent vertices such that $x\preceq a\prec b\preceq y$ and $V(a,b)$ is empty, then add an edge between $a$ and $b$.
This does not affect much since the edge $ab$ is not in any crossings, and we will remove these edges later.
Refer to Figure~\ref{fig:S-empty} for the following.

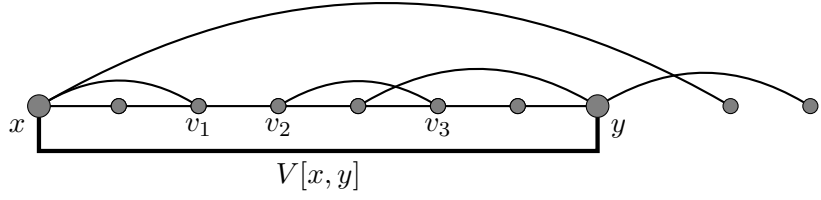
\begin{figure}[t]
\centering
\begin{tikzpicture}[scale=.7, every node/.style={inner sep=2, outer sep=0}, vtx/.style={draw, circle, fill=gray}]
  \node[vtx, inner sep=3, label=below left:$x$] (x) at (-6.5,0) {};
  \node[vtx, inner sep=3, label=below right:$y$] (y) at (4,0) {};
  \node (x1) at (-6.5,-.85) {};
  \node (y1) at (4,-.85) {};
  \node[label=below:{$V[x,y]$}] at (-1.25,-.85) {};
  \draw[ultra thick] (x) -- (x1.center) -- (y1.center) -- (y);
  \node[vtx] (a1) at (6.5,0) {};
  \node[vtx] (a2) at (8,0) {};
  \draw[thick] (x) to [bend left=30] (a1);
  \draw[thick] (y) to [bend left=30] (a2);
  \node[vtx] (vS) at (-5,0) {};
  \node[vtx, label=below:$v_1$] (vR) at (-3.5,0) {};
  \node[vtx, label=below:$v_2$] (e) at (-2,0) {};
  \node[vtx] (vL) at (2.5,0) {};
  \node[vtx] (b2) at (-.5,0) {};
  \node[vtx, label=below:$v_3$] (b3) at (1,0) {};
  \draw[thick] (x) -- (y);
  \draw[thick] (e) to [bend left=30] (b3);
  \draw[thick] (b2) to [bend left=30] (y);
  \draw[thick] (x) to [bend left=30] (vR);
  \node[vtx] at (vS) {};
  \node[vtx] at (vR) {};
  \node[vtx] at (vL) {};
  \node[vtx] at (b2) {};
  \node[vtx] at (b3) {};
  \node[vtx] at (e) {};
\end{tikzpicture}
\caption{The vertices $x=v_0\prec v_1\prec v_2\prec v_3\prec v_4=y$.}
\label{fig:S-empty}
\end{figure}

Set $v_0\coloneqq x$ and let $v_1$ be the largest neighbour of $x$ in $V[x,y)$.
The vertex $v_1$ exists since there is an uncoloured interior vertex and because of the added edges.
Now suppose that we have already defined vertices $x=v_0\prec v_1\prec\cdots\prec v_i\prec y$.
Then let $v_{i+1}$ be the largest neighbour of a vertex in $V(v_{i-1},v_i]$, subject to satisfying $v_{i+1}\in V(v_i,y]$.
Again, $v_{i+1}$ exists because of the added edges.
If $v_{i+1}=y$ then we are done; otherwise continue the process.
In the end we have defined vertices $x=v_0\prec v_1\prec\cdots\prec v_{k+1}=y$, with $k\geq 1$.

Now delete the added edges so that we are again considering the original ordered graph, and colour all vertices $v_i$, with $i\in\{1,\ldots,k\}$, the colour $\phi(x)$.
Call the new colouring $\phi'$.
There are a few things to observe.
First, for each $i\in\{0,\ldots,k\}$, if $V(v_i,v_{i+1})$ is non-empty, then $V[v_i,v_{i+1}]$ is valid.
Second, each colour class $X$ of $\phi'$ induces a capped subgraph (note that by Lemma~\ref{lem:extends}, it suffices to observe that the subgraph induced by $X\cap V[x,y]$ is capped).

There are two colours besides $\phi(x)$, say ``blue'' and ``green''.
We colour all vertices which are strongly reachable from a segment $V[v_i,v_{i+1}]$ ``blue'' if $i$ is odd and ``green'' if $i$ is even.
If necessary, we permute the colours ``blue'' and ``green'' so that the strongly reachable vertices of $V[v_{k-1},v_k]$ are not given colour $\phi(y)$.
Call the new colouring $\phi''$.
It can be verified that for each segment $V[v_i,v_{i+1}]$, no interior vertex is adjacent to an exterior vertex of the same colour (note that all such interior vertices are strongly reachable from $V[v_i,v_{i+1}]$ by Lemma~\ref{lem:extInt} and are therefore actually coloured).
By Lemmas~\ref{lem:induceCapped} and~\ref{lem:extends}, each colour class of $\phi''$ induces a capped subgraph.

So, by recursion, for each $i\in\{0,\ldots,k\}$ such that $V(v_i,v_{i+1})$ is non-empty, we can find a $V[v_i,v_{i+1}]$-extension of $\phi''$.
By Lemma~\ref{lem:extends}, each colour class of this extension induces a capped subgraph, and so we may continue.
This completes the case that $S(x,y)$ is empty.

\emph{Case 2.}\enspace
The set $S(x,y)$ is non-empty.

It is convenient to call two segments \emph{internally disjoint} if they have at most one vertex in common, which is an end of both segments.
We define sets $\cL$ and $\cR$ of valid segments so that
\begin{enumeratei}
\item every pair of segments in $\cL$ ($\cR$, respectively) is internally disjoint,
\item $V[x,y]=\bigcup_{V[a,b]\in\cL\cup\cR}V[a,b]$, and
\item no segment in $\cL$ ($\cR$, respectively) has an interior vertex which is right-reachable (left-reachable, respectively) from $V[x,y]$.
\end{enumeratei}
Let $s_1\prec\cdots\prec s_k$ denote the strongly reachable vertices, with $k\geq 1$, and set $s_0\coloneqq x$ and $s_{k+1}\coloneqq y$.
For the following definitions, refer to Figure~\ref{fig:LandR}.

\begin{figure}[t]
\centering
\begin{tikzpicture}[scale=.7, every node/.style={inner sep=2, outer sep=0}, vtx/.style={draw, circle, fill=gray}, arr/.style={->, thick}]
  \def\dif{1.5}
  \def\bot{.8}
  \def\L{.5}
  \def\H{.65}
  \def\save{.1}
  \node[vtx, inner sep=3, label=left:{$x=s_0$}, label=below:$L$, fill=red] (x) at (-7.5,0) {};
  \node[vtx, inner sep=3, label=right:{$s_2=y$}, label=below:$R$, fill=green] (y) at (7.5,0) {};
  \node[vtx, label=below:$s_1$, fill=red] (s1) at (0,0) {};
  \node[vtx, label=below:$L$] (v1) at (-7.5+\dif,0) {};
  \node[vtx, label=below:$R$] (v2) at (-7.5+2*\dif,0) {};
  \node[vtx, label=below:$R$] (v3) at (-7.5+3*\dif,0) {};
  \node[vtx, label=below:$L$] (v4) at (-7.5+4*\dif,0) {};
  \node[vtx, label=below:$L$] (v5) at (7.5-4*\dif,0) {};
  \node[vtx, label=below:$R$] (v6) at (7.5-3*\dif,0) {};
  \node[vtx, label=below:$L$] (v7) at (7.5-2*\dif,0) {};
  \node[vtx, label=below:$R$] (v8) at (7.5-\dif,0) {};
  \draw[ultra thick] (v1) -- ++ (0,\H) -- ++ (3*\dif,0) -- (v4);
  \draw[ultra thick] {(x) ++ (0,-\bot)} -- ++ (0,-\L) -- ++ (2*\dif,0) -- ++ (0,\L);
  \draw[ultra thick] {(v4) ++ (0,-\bot)} -- ++ (0,-\L) -- ++ (\dif-\save,0) -- ++ (0,\L);
  \draw[ultra thick] {(s1) ++ (\save,-\bot)} -- ++ (0,-\L) -- ++ (2*\dif-2*\save, 0) -- ++ (0,\L);
  \draw[ultra thick] {(v6) ++ (\save,-\bot)} -- ++ (0,-\L) -- ++ (2*\dif,0) -- ++ (0,\L);
  \draw[ultra thick] (v7) -- ++ (0,\H) -- ++ (2*\dif, 0) -- (y);
  \node[label=above:$\cR$] (R) at (0,2) {};
  \draw[arr] (R) -- (-2.5*\dif,\H+.2);
  \draw[arr] (R) -- (4*\dif,\H+.2);
  \node[label=below:$\cL$] (L) at (0,-2-\bot) {};
  \draw[arr] (L) -- (-4*\dif,-\bot-\L-.2);
  \draw[arr] (L) -- (-.5*\dif,-\bot-\L-.2);
  \draw[arr] (L) -- (\dif,-\bot-\L-.2);
  \draw[arr] (L) -- (3*\dif,-\bot-\L-.2);
\end{tikzpicture}
\caption{A depiction of the sets $L[x,y)\cup R(x,y]$, $\cL$, and $\cR$.}
\label{fig:LandR}
\end{figure}
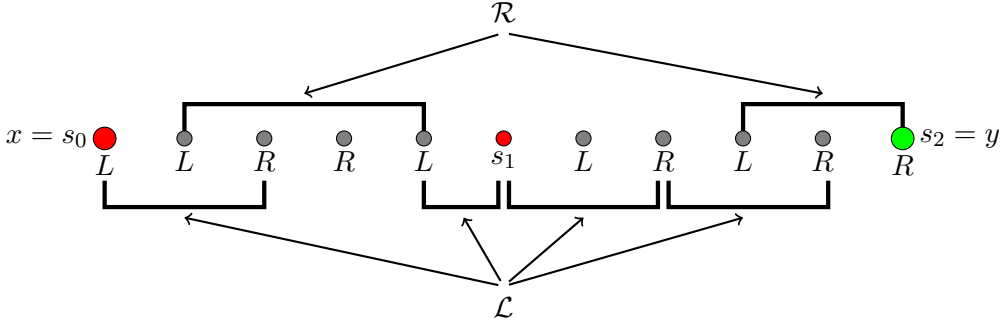

First, for each $r\in V(s_0,s_1)$ which is right-reachable, let $\ell_1,\ell_2\in V[s_0,s_1]$ be left-reachable vertices with $\ell_1\prec r\prec\ell_2$; choose $\ell_1$ as large as possible and $\ell_2$ as small as possible.
We add the segment $V[\ell_1,\ell_2]$ to $\cR$ (if it was not already in $\cR$).
Note that $\ell_1$ and $\ell_2$ exist since $k\geq 1$.
Furthermore, $V[\ell_1,\ell_2]$ is valid; either $\ell_2$ is right-reachable and $V[\ell_1,\ell_2]$ satisfies condition~(i) of Lemma~\ref{lem:valid}, or $\ell_2\prec s_1$ and $V[\ell_1,\ell_2]$ satisfies condition~(ii) of Lemma~\ref{lem:valid}.

Now we add segments to $\cL$ to ``cover the gaps in $V[s_0,s_1]$ between segments in $\cR$''.
We add one segment to $\cL$ for each maximal segment $V[\ell_2',\ell_1']\subseteq V[s_0,s_1]$ which is internally disjoint from all segments in $\cR$ (observe that either $\ell_2'=s_0$ or $\ell_2'$ is the largest end of a segment in $\cR$, and similarly for $\ell_1'$).
Let $r_2$ be the smallest right-reachable vertex in $V[s_0,s_1]$ with $\ell_1'\preceq r_2$.
Such a vertex exists.
We add the segment $V[\ell_2',r_2]$ to $\cL$.
The segment $V[\ell_2',r_2]$ is valid since it satisfies condition~(i) of Lemma~\ref{lem:valid}.

This defines the segments in $\cL\cup\cR$ which are contained in $V[s_0,s_1]$.
For each $i\in\{1,\ldots,k-1\}$, we use the same procedure to define the segments in $\cL\cup\cR$ which are contained in $V[s_i,s_{i+1}]$.
For $V[s_k,s_{k+1}]$, however, this approach may not work, as condition~(ii) of Lemma~\ref{lem:valid} need not hold.
So within this segment we switch the roles of $\cL$ and $\cR$ (first add segments to $\cL$ by considering vertices $\ell\in V(s_k,s_{k+1})$ which are left-reachable, and so on).
The segments defined this way are valid by conditions~(i) and~(iii) of Lemma~\ref{lem:valid}.
This completes the definitions of $\cL$ and $\cR$; it can be verified that they satisfy conditions~(i)--(iii) above.
Furthermore, we have the following claim about where ``interesting'' edges lie.

\begin{claim}
\label{clm:edges}
If $V[u,v]$ is a segment contained in $V[x,y]$ such that $uv\in E(G)$ and $V(u,v)$ contains an end of a segment in $\cL\cup\cR$, then $u$ is not interior to any segment in $\cL$ and $v$ is not interior to any segment in $\cR$.
\end{claim}

\begin{proof}
Up to the symmetry obtained by reversing $\prec$, it suffices to prove that $v$ is not interior to any segment in $\cR$.
So, going for a contradiction, suppose that $v\in V(a,b)$ for some segment $V[a,b]\in\cR$.
Since $a$ is left-reachable, there are crossing sequences from $y$ to $a$ and from $u$ to $v$.
If $u\prec a$, then by Lemma~\ref{lem:almostTran}, there is a crossing sequence from $y$ to $v$ and $v$ is left-reachable.
This contradicts condition~(iii) for $\cR$.

So $a\preceq u$.
Since $V(u,v)$ contains an end $r$ of a segment in $\cL\cup\cR$ and segments in $\cR$ are internally disjoint, $r$ must be an end of a segment in $\cL$.
In fact $r$ must be the larger end of a segment in $\cL$; so $x\prec u\prec r\prec v$ and $r$ is right-reachable from $V[x,y]$.
Then, because there are crossing sequences from $u$ to $v$ and from $r$ to $x$, by Lemma~\ref{lem:almostTran} there is a crossing sequence from $u$ to $x$ and $u$ is right-reachable.
This contradicts condition~(iii) for $\cL$.
\end{proof}

Now we complete the algorithm for colouring.
Suppose that so far we have a colouring $\hat{\phi}$, where originally we take $\hat{\phi}\coloneqq\phi$.
Then take an arbitrary segment $V[a,b]\in\cL\cup\cR$ which has not yet been considered.
If $a$ or $b$ is uncoloured, then colour it $L$ if it is in $L[x,y)$ and $R$ if it is in $R(x,y]$.
Furthermore, colour all vertices in $S(a,b)$ colour $L$ if $V[a,b]\in\cL$ and colour $R$ if $V[a,b]\in\cR$.
Call the new colouring $\hat{\phi}'$.
We will show that $\hat{\phi}'$ is a $V[a,b]$-precolouring.
Thus we can recursively find a $V[a,b]$-extension of $\hat{\phi}'$.
Take the extension, uncolour any vertices which are interior to a segment in $\cL\cup\cR$ which has not yet been considered, and continue the process by updating $\hat{\phi}$.
By condition~(ii) above (that $\cL\cup\cR$ ``covers'' $V[x,y]$) and by Lemma~\ref{lem:extends}, after every segment in $\cL\cup\cR$ has been considered, the updated $\hat{\phi}$ is a $V[x,y]$-extension of $\phi$; so we then return~$\hat{\phi}$.
Observe that the running time is polynomial.

Consider a step in the algorithm where we began with $\hat{\phi}$ and then obtained $\hat{\phi}'$ by considering $V[a,b]\in\cL\cup\cR$.
We just need to show the following.

\begin{claim}
\label{claim:isPrecolouring}
The partial $3$-colouring $\hat{\phi}'$ is a $V[a,b]$-precolouring.
\end{claim}

\begin{proof}
Up to the symmetry obtained by reversing $\prec$, we can assume that $V[a,b]\in\cR$.
We know that $\hat{\phi}'$ colours $S(a,b)\cup\{x,y\}$ and colours no other interior vertices of $V[a,b]$, and that $\hat{\phi}'$ gives every vertex in $S(a,b)$ the colour $R$.
So, by Lemmas \ref{lem:induceCapped} and~\ref{lem:extends}, it suffices to show that
\begin{enumeratei}
\item no exterior vertex $u$ of $V[a,b]$ which is adjacent to an interior vertex $v$ of $V[a,b]$ receives colour $R$, and
\item for each colour class $X$ of $\hat{\phi}'$, the set $X\setminus V(a,b)$ induces a capped subgraph.
\end{enumeratei}

First we prove (i) by contradiction.
We know that $u\in V[x,y)$ since otherwise $v$ would be left-reachable.
Then by Claim~\ref{clm:edges} in fact $u\in V(b,y)$.
If $u$ was coloured due to being interior to a segment $V[a',b']\in\cL\cup\cR$, then, since $V[a,b]$ and $V[a',b']$ are internally disjoint, $u\in S(a',b')$.
Thus $V[a',b']\in\cR$, but this is a contradiction to Claim~\ref{clm:edges} (where the roles of $u$ and $v$ are reversed).
So $u$ must have been coloured due to being the end of a segment in $\cL\cup\cR$.
So $u$ is right-reachable but not left-reachable.
But there is a left-reachable vertex $\ell$ with $v\prec\ell\prec u$ (possibly $\ell=b$), which contradicts the fact that $u$ is not left-reachable.

Now we prove (ii).
By induction, we can assume that every colour class of $\hat{\phi}$ induces a capped subgraph.
So it suffices to prove that for any $v\in\{a,b\}$ which is newly coloured by $\hat{\phi}'$, there is no edge $e = uv$ which is in a crossing with an edge $f$ in the subgraph induced by $X\setminus V(a,b)$.
Suppose for the sake of contradiction that $e$ and $f$ do exist.
Since $v$ is newly coloured, $v$ is not in any segment in $\cL\cup\cR$ that has already been considered.
This implies, by considering why $u$ was coloured, that if $\hat{\phi}'(u)=R$ ($L$, respectively), then $u$ is right-reachable (left-reachable, respectively).

Suppose that $\hat{\phi}'(u)=R$.
By possibly switching the labels of $u$ and $v$, we have adjacent vertices $v'\prec u'$ such that both $v'$ and $u'$ are right-reachable but not left-reachable.
So no vertex in $V[v',u']$ is left-reachable, and thus there is a segment $V[a', b']\in\cL\cup\cR$ which contains both $v'$ and $u'$.
Since $f$ is crossing with $e=v'u'$ and the ends of $f$ are already coloured by $\hat{\phi}$, the segment $V[a',b']$ was already considered.
But this is a contradiction since we showed that $v\notin V[a',b']$.
A symmetric argument works in the case that $\hat{\phi}'(u)=L$.
This completes the proof of Claim~\ref{claim:isPrecolouring}.
\end{proof}

Claim~\ref{claim:isPrecolouring} above completes the proof of Proposition~\ref{prop:extend}.
\end{proof}

\begin{proof}[Proof of Proposition \ref{prop:partition}]
Let $(G',{\prec'})$ be the ordered graph obtained from $(G,{\prec})$ by adding a new smallest vertex $x$, a new largest vertex $y$, and the edge $xy$.
Let $\phi'$ be any partial colouring which only colours $x$ and $y$.
Then $(G',{\prec'})$ is $\cH$-free, $V[x,y]$ is a valid segment, and $\phi'$ is a $V[x,y]$-precolouring.
Therefore, by Proposition~\ref{prop:extend}, we can find (in polynomial time) a $3$-colouring of $(G,{\prec})$ such that each colour class induces a capped subgraph.
\end{proof}

\section{Colouring capped graphs}
\label{sec:capped}

This section is devoted to the proof of Theorem~\ref{thm:capped} on colouring capped graphs.
The proof relies on decompositions; a \emph{decomposition} of an ordered graph $(G,{\prec})$ is a collection of subgraphs such that every edge of $G$ belongs to exactly one subgraph in the collection.
For a graph $G$ and a set $F\subseteq E(G)$, we write $G[F]$ and $G-F$ for the graph obtained from $G$ by keeping/deleting (respectively) the edges in $F$.

\begin{proposition}
\label{prop:mainDecomp}
There is a polynomial-time algorithm which takes in a capped graph\/ $(G,{\prec})$ and returns its clique number\/ $\omega$ and a decomposition of\/ $(G,{\prec})$ into\/ $\omega-1$ triangle-free capped graphs.
If\/ $(G,{\prec})$ is additionally ordered-hole-free, then so is each graph in the decomposition.
\end{proposition}

\begin{proof}
If $G$ is triangle-free, then we return its clique number ($1$ or $2$) and the decomposition consisting of $(G,{\prec})$ itself.
Thus assume henceforth that $G$ is not triangle-free.
We say that an edge $uv\in E(G)$ with $u\prec v$ is \emph{triangle-crossed} if $v$ belongs to a triangle with vertices $x$ and $y$ such that $x,y\preceq u$.
Let $F$ be the set of all edges that are not triangle-crossed (see Figure~\ref{fig:decomp}).

\begin{figure}[t]
\centering
\begin{tikzpicture}[yscale=.8, xscale=.65, every node/.style={inner sep=2, outer sep=0, draw, circle, fill=gray}]
  \node[label=below:$6$] (A6) at (7.5,0) {};
  \node[label=below:$5$] (A5) at (4.5,0) {};
  \node[label=below:$4$] (A4) at (1.5,0) {};
  \node[label=below:$3$] (A3) at (-1.5,0) {};
  \node[label=below:$2$] (A2) at (-4.5,0) {};
  \node[label=below:$1$] (A1) at (-7.5,0) {};
  \draw[ultra thick] (A1) to (A2);
  \draw[ultra thick] (A2) to (A3);
  \draw[gray!80, ultra thick] (A3) to (A4);
  \draw[gray!40, ultra thick] (A4) to (A5);
  \draw[gray!40, ultra thick] (A5) to (A6);
  \draw[ultra thick] (A1) to[bend left=30] (A6);
  \draw[ultra thick] (A1) to[bend left=30] (A4);
  \draw[ultra thick] (A1) to[bend left=30] (A5);
  \draw[gray!80, ultra thick] (A2) to[bend left=30] (A4);
  \draw[gray!80, ultra thick] (A2) to[bend left=30] (A5);
  \draw[gray!80, ultra thick] (A2) to[bend left=30] (A6);
  \draw[gray!80, ultra thick] (A3) to[bend left=30] (A5);
  \foreach\i in {1,...,6}{\node at (A\i) {};}
\end{tikzpicture}
\caption{The decomposition from Proposition~\ref{prop:mainDecomp}, where darker edges are removed first.}
\label{fig:decomp}
\end{figure}
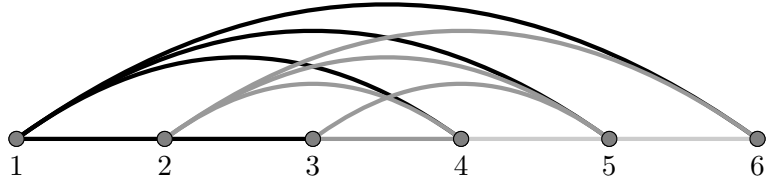

If $a,b,c$ are vertices of $G$ such that $a\prec b\prec c$, $ac,bc\in E(G)$, and $ac$ is triangle-crossed, then $bc$ is triangle-crossed.
This implies that $(G[F],{\prec})$ is capped and is ordered-hole-free if $(G,{\prec})$ is ordered-hole-free.
Furthermore, if vertices $a,b,c$ with $a\prec b\prec c$ form a triangle in $G$, then $bc$ is triangle-crossed.
So $G[F]$ is triangle-free.
Moreover, we have the following.

\begin{claim}
The graph $(G-F,{\prec})$ is capped and has clique number exactly one less than the clique number of $(G,{\prec})$.
Furthermore, if $(G,{\prec})$ is ordered-hole-free, then so is $(G-F,{\prec})$.
\end{claim}

\begin{proof}
Let $\omega$ denote the clique number of $(G,{\prec})$.
If $a,b,c$ are vertices of $G$ such that $a\prec b\prec c$, $ab,ac\in E(G)$, and $ab$ is triangle-crossed, then $ac$ is triangle-crossed.
So $(G-F,{\prec})$ is capped and is ordered-hole-free if $(G,{\prec})$ is.
Furthermore, the clique number of $(G-F,{\prec})$ is at least $\omega-1$, because every edge of a clique in $(G,{\prec})$ that is not incident to the smallest vertex of the clique is triangle-crossed.
Now, suppose that $Q\subseteq V(G)$ is a clique in $G-F$, and let $u$ and $v$ be the two smallest vertices of $Q$, with $u\prec v$.
Then $v$ is in a triangle of $G$ with vertices $x$ and $y$ such that $x\prec y\preceq u$.
It follows that $(Q\setminus\{u\})\cup\{x,y\}$ is a clique in $G$.
This shows that the clique number of $(G-F,{\prec})$ is at most $\omega-1$, as desired.
\end{proof}

To conclude, the algorithm proceeds by continuing with $(G-F,{\prec})$.
The clique number is the number of subgraphs in the decomposition plus one.
\end{proof}

\begin{proof}[Proof of Theorem \ref{thm:capped}]
Let $\omega\geq 2$ be the clique number of $G$, and let $\{(G_i,{\prec})\}_{1\leq i<\omega}$ be a decomposition of $(G,{\prec})$ into $\omega-1$ triangle-free capped subgraphs as in Proposition~\ref{prop:mainDecomp}.
Fix an index $i$ with $1\leq i<\omega$.
If $(G,{\prec})$ is ordered-hole-free, then let $F_i=\emptyset$, and otherwise let $F_i$ be the set of edges of $(G_i,{\prec})$ which are not crossed in $(G_i,{\prec})$.
An ordered graph is \emph{outerplanar} if it has no crossing pair of edges.

\begin{claim}
\label{claim:outerplanar}
The ordered graph $(G[F_i],{\prec})$ is outerplanar, and $(G_i-F_i,{\prec})$ is both capped and ordered-hole-free.
\end{claim}

\begin{proof}
We can assume that $(G,{\prec})$ is not ordered-hole-free.
That $(G[F_i],{\prec})$ is outerplanar is clear from the definition of $F_i$.
If $a,b,c$ are vertices such that $a\prec b\prec c$, $ab,ac\in E(G_i)$, and $ab$ is crossed, then $ac$ is crossed.
This implies that the ordered graph $(G_i-F_i,{\prec})$ is capped and every ordered hole in it is an ordered hole in $(G_i,{\prec})$.
Suppose for the sake of contradiction that vertices $c_1\prec\cdots\prec c_k$ induce an ordered hole in $(G_i-F_i,{\prec})$ and thus in $(G_i,{\prec})$.
Since the edge $c_{k-2}c_{k-1}$ is crossed, there is an edge $xy\in E(G_i)$ with $x\prec c_{k-2}\prec y\prec c_{k-1}$.
It follows that $x\prec c_1$, as $c_1,\ldots,c_k$ induce an ordered hole in $(G_i,{\prec})$.
We conclude that $x$, $c_{k-1}$, and $c_k$ form a triangle in $G_i$.
This contradiction shows that $(G_i-F_i,{\prec})$ is ordered-hole-free.
\end{proof}

\begin{claim}
\label{claim:triangleFree}
There is a $4$-colouring of $G_i-F_i$, which can be computed in polynomial time.
\end{claim}

\begin{proof}
We just use the fact that $(G_i-F_i,{\prec})$ is triangle-free, capped, and ordered-hole-free.
For each component of $G_i-F_i$, we claim that any level of any breadth-first search tree which is rooted at the smallest vertex according to $\prec$ induces a bipartite subgraph.
This suffices to complete the proof, as we can reuse colours at every second level.

Suppose for the sake of contradiction that $p$ is the smallest vertex of a component and $C$ is an induced odd cycle which is contained in a level.
Since $(G_i-F_i,{\prec})$ is triangle-free and ordered-hole-free, there are $ac,bd\in E(C)$ such that $a\prec b\prec c\prec d$.
So $ad\in E(C)$, and none of the edges $ac,bd,ad$ are crossing with any other edge of $C$.
It follows that $V(C)\setminus\{a,d\}$ induces a path $bv_1\cdots v_tc$ with $b\prec v_1\prec\cdots\prec v_t\prec c$, for some positive integer $t$.

Let $P$ be a shortest path from $v_1$ to $p$ in $G_i-F_i$, and let $v_1'$ be the vertex adjacent to $v_1$ in $P$.
If $a\preceq v_1'\preceq d$ then we obtain a contradiction by finding a path which is shorter than $P$ from either $a$ or $d$ to $p$, using the fact that $(G_i-F_i,{\prec})$ is capped.
Otherwise, if $v_1'\prec a$, then $\{v_1',v_1,\ldots,v_t,c\}$ contains a triangle or an ordered hole, which is a contradiction.
In the final case that $d\prec v_1'$, the vertices $b,v_1,v_1'$ form a triangle, which is again a contradiction.
\end{proof}

Let $\phi_i$ be a $4$-colouring of $G_i-F_i$ from the last claim, for $1\leq i<\omega$.
Let $F=\bigcup_{i=1}^{\omega-1}F_i$.
If $(G,{\prec})$ is ordered-hole-free, then $F=\emptyset$ and the mapping $v\mapsto(\phi_1(v),\ldots,\phi_{\omega-1}(v))$ is a $4^{\omega-1}$-colouring of $G$.
Otherwise, since every $n$-vertex outerplanar graph has at most $2n-3$ edges, every $n$-vertex subgraph of $G[F]$ has at most $(2n-3)(\omega-1)$ edges, for any $n\geq 2$.
So every non-empty subgraph of $G[F]$ has a vertex of degree less than $4(\omega-1)$, and thus there exists a $4(\omega-1)$-colouring $\psi$ of $G[F]$.
Now, the mapping $v\mapsto(\phi_1(v),\ldots,\phi_{\omega-1}(v),\psi(v))$ is a $4^\omega(\omega-1)$-colouring of $G$.
\end{proof}

\section{Colouring $X$-free ordered graphs}
\label{sec:bananas}

In this section, we prove Theorem~\ref{thm:X-free} that $X$-free ordered graphs are $\chi$-bounded, using a surprising connection to a recent theorem of Scott and Seymour~\cite{banana20}.
Recall that $X$ is the ordered graph with four vertices and two crossing edges, illustrated in Figure~\ref{fig:banana_X} (right).
It is immediate from the definition that capped graphs are $X$-free.
Moreover, various natural classes of ordered intersection graphs are also $X$-free.
Specifically, an \emph{outerstring graph} is the intersection graph of a collection of curves in a half-plane which each have one endpoint on the boundary of that half-plane.
It is easy to see that outerstring graphs are $X$-free when equipped with the ordering of their endpoints along the boundary.
Outerstring graphs contain circle graphs and interval filament graphs.
Rok and Walczak~\cite{RWouterstring} proved that the class of outerstring graphs is $\chi$-bounded.
Theorem~\ref{thm:X-free} yields this results as a direct corollary.

We now prove that the underlying (unordered) graph of any $X$-free ordered graph forbids every induced subdivision of a certain banana.
A \emph{banana} is a graph which can be obtained from the disjoint union of paths by identifying one end of each path to a new vertex $s$, and the other end of each path to a different new vertex $t$.
A \emph{subdivision} of a graph $H$ is a graph obtained from $H$ by replacing edges of $H$ with internally disjoint paths between their ends.
Scott and Seymour~\cite{banana20} proved the following.

\begin{theorem}[Scott and Seymour~\cite{banana20}]
\label{thm:SS}
For any banana\/ $B$, the class of graphs excluding all subdivisions of\/ $B$ as induced subgraphs is\/ $\chi$-bounded.
\end{theorem}

Let $B_4$ denote the banana depicted in Figure~\ref{fig:banana_X} (left); it is obtained from three $4$-edge paths.
Theorem~\ref{thm:SS} together with the following proposition yields Theorem~\ref{thm:X-free} and, consequently, another proof that outerstring graphs and capped graphs are $\chi$-bounded.

\begin{proposition}
\label{prop:banana}
For any\/ $X$-free ordered graph\/ $(H,{\prec})$, the graph\/ $H$ contains no subdivision of\/ $B_4$ as an induced subgraph.
\end{proposition}

\begin{proof}
First we prove that $B_4$ itself is not the underlying graph of an $X$-free ordered graph.
Suppose for a contradiction that there is a vertex ordering $\prec$ such that $(B_4,{\prec})$ is $X$-free.
Any cyclic reordering of $\prec$ also yields an $X$-free ordering.
So, as none of the edges $sx_1,sy_1,sz_1$ cross any of the edges $tx_3,ty_3,tz_3$, after cyclically reordering the vertices, we can assume that each of $s,x_1,y_1,z_1$ is less than each of $t,x_3,y_3,z_3$.

Consider the three paths $x_1,x_2,x_3$, $y_1,y_2,y_3$, and $z_1,z_2,z_3$.
No edge of one of these paths may cross an edge of another.
So, for instance, we cannot have that $x_1\prec y_1\prec x_3\prec y_3$.
Then, up to relabelling these three paths, we can assume that $x_1\prec y_1\prec z_1\prec z_3\prec y_3\prec x_3$.
Now we examine where the vertex $y_2$ lies in the ordering.
For the same reason, either $x_1\prec y_2\prec z_1$ or $z_3\prec y_2\prec x_3$; without loss of generality, we can assume the first case.
But now one of the two edges $sx_1$ or $sz_1$ is crossing with the edge $y_2y_3$, a contradiction.

Now suppose that $(G,{\prec})$ is an $X$-free ordered graph and $v$ is a degree-$2$ vertex which is adjacent to non-adjacent vertices $u$ and $w$.
We will show that, where $H$ is the graph obtained from $G$ by deleting $v$ and adding the edge $uw$, and $\prec_H$ is the restriction of $\prec$ to $V(H)$, the ordered graph $(H,{\prec_H})$ is $X$-free.
This will complete the proof.
Up to cyclically reordering the vertices and renaming $u$ and $w$, the only case to consider is that $u\prec v\prec w$.
If $(H,{\prec_H})$ is not $X$-free, then there must be a copy of $X$ containing the edge $uw$.
But then either $uv$ or $vw$ is in a copy of $X$ in $(G,{\prec})$, which is a contradiction.
\end{proof}

\begin{figure}[t]
\centering
\begin{tikzpicture}[scale=1.25, every node/.style={inner sep=2, outer sep=0}, vtx/.style={draw, circle, fill=gray}]
  \coordinate (s) at (30:2cm) {};
  \coordinate (x1) at (60:2cm) {};
  \coordinate (y1) at (90:2cm) {};
  \coordinate (z1) at (120:2cm) {};
  \coordinate (x2) at (150:2cm) {};
  \coordinate (y2) at (180:2cm) {};
  \coordinate (z2) at (210:2cm) {};
  \coordinate (x3) at (240:2cm) {};
  \coordinate (y3) at (270:2cm) {};
  \coordinate (z3) at (300:2cm) {};
  \coordinate (t) at (330:2cm) {};
  \path[name path=s-y1] (s)--(y1);
  \path[name path=s-z1] (s)--(z1);
  \path[name path=x1-x2] (x1)--(x2);
  \path[name path=y1-y2] (y1)--(y2);
  \path[name path=z1-z2] (z1)--(z2);
  \path[name path=x2-x3] (x2)--(x3);
  \path[name path=y2-y3] (y2)--(y3);
  \path[name path=z2-z3] (z2)--(z3);
  \path[name path=x3-t] (x3)--(t);
  \path[name path=y3-t] (y3)--(t);
  \path[name intersections={of=s-z1 and x1-x2, by=y1p}];
  \path[name intersections={of=x1-x2 and y1-y2, by=z1p}];
  \path[name intersections={of=y1-y2 and z1-z2, by=x2p}];
  \path[name intersections={of=z1-z2 and x2-x3, by=y2p}];
  \path[name intersections={of=x2-x3 and y2-y3, by=z2p}];
  \path[name intersections={of=y2-y3 and z2-z3, by=x3p}];
  \path[name intersections={of=z2-z3 and x3-t, by=y3p}];
  \path[name intersections={of=s-y1 and x1-x2, by=x1y1}];
  \path[name intersections={of=s-z1 and y1-y2, by=y1z1}];
  \path[name intersections={of=x1-x2 and z1-z2, by=z1x2}];
  \path[name intersections={of=y1-y2 and x2-x3, by=x2y2}];
  \path[name intersections={of=z1-z2 and y2-y3, by=y2z2}];
  \path[name intersections={of=x2-x3 and z2-z3, by=z2x3}];
  \path[name intersections={of=y2-y3 and x3-t, by=x3y3}];
  \path[name intersections={of=z2-z3 and y3-t, by=y3z3}];
  \draw[line width=3pt] (s)--(x1)--(x1y1)--(y1)--(y1z1)--(z1)--(z1x2)--(x2)--(x2y2)--(y2)--(y2z2)--(z2)--(z2x3)--(x3)--(x3y3)--(y3)--(y3z3)--(z3)--(t);
  \draw[line width=3pt] (t)--(y3p)--(x3p)--(z2p)--(y2p)--(x2p)--(z1p)--(y1p)--(s);
  \node at (30:2.3cm) {$s$};
  \node at (60:2.3cm) {$x_1$};
  \node at (90:2.3cm) {$y_1$};
  \node at (120:2.3cm) {$z_1$};
  \node at (150:2.3cm) {$x_2$};
  \node at (180:2.3cm) {$y_2$};
  \node at (210:2.3cm) {$z_2$};
  \node at (240:2.3cm) {$x_3$};
  \node at (270:2.3cm) {$y_3$};
  \node at (300:2.3cm) {$z_3$};
  \node at (330:2.3cm) {$t$};
  \draw[red, thick] (s)--(x1)--(x2)--(x3)--(t);
  \draw[red, thick] (s)--(y1)--(y2)--(y3)--(t);
  \draw[red, thick] (s)--(z1)--(z2)--(z3)--(t);
  \node[vtx] at (s) {};
  \node[vtx] at (x1) {};
  \node[vtx] at (x2) {};
  \node[vtx] at (x3) {};
  \node[vtx] at (y1) {};
  \node[vtx] at (y2) {};
  \node[vtx] at (y3) {};
  \node[vtx] at (z1) {};
  \node[vtx] at (z2) {};
  \node[vtx] at (z3) {};
  \node[vtx] at (t) {};
\end{tikzpicture}
\caption{$B_4$ as an induced subgraph of a polygon visibility graph.}
\label{fig:B4-poly}
\end{figure}
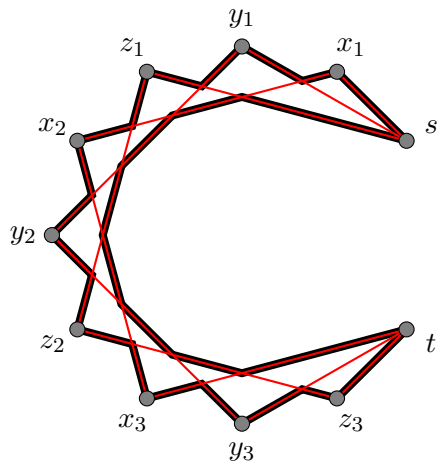

Scott and Seymour~\cite{banana20} actually proved Theorem~\ref{thm:SS} for \emph{banana trees}, which are obtained from trees by replacing each edge with a banana so that the ends of the edge are $s$ and $t$ (the high-degree vertices of the banana).
It is natural to ask if this theorem could be applied directly to polygon visibility graphs.
However, this is not possible; it is not difficult to see that every banana tree is an induced subgraph of a polygon visibility graph.
Figure~\ref{fig:B4-poly} illustrates this for the banana $B_4$.

\section*{Acknowledgement}

We thank Bodhayan Roy for sharing the problem of whether polygon visibility graphs are $\chi$-bounded.

\bibliographystyle{plainurl}
\bibliography{visibility}

\end{document}